\documentclass[11pt, final]{amsart}
\usepackage{amsmath,amsthm,amssymb,amsfonts}
\usepackage{xcolor}
\usepackage[colorlinks=false, final]{hyperref}
\usepackage{verbatim}
\usepackage{enumitem}
\usepackage{xspace}
\usepackage{multicol}
\usepackage{geometry}
\setlist[enumerate,1]{label=(\arabic*)} 
\DeclareMathOperator{\FI}{FI}
\DeclareMathOperator{\FA}{FA}
\DeclareMathOperator{\FLA}{FLA}
\DeclareMathOperator{\FAd}{FE}
\DeclareMathOperator{\FLAd}{FLE}
\theoremstyle{plain} 
\newtheorem{thm}{Theorem}[section] 
\newtheorem{theorem}[thm]{Theorem} 
\newtheorem{proposition}[thm]{Proposition}
\newtheorem{lemma}[thm]{Lemma}
\newtheorem{corollary}[thm]{Corollary}
\theoremstyle{definition} 
\newtheorem{definition}[thm]{Definition}
\theoremstyle{remark} 
\newtheorem{remark}[thm]{Remark} 
\newtheorem{example}[thm]{Example}
\newtheorem{question}[thm]{Question}
\counterwithin{equation}{section}
\counterwithin{figure}{section}

\makeatletter
\providecommand\@dotsep{5}
\makeatother
\title{Forbidden configurations for coherency}
\date{\today}
\author[V. Gould]{Victoria Gould}
\address{University of York}
\email{victoria.gould@york.ac.uk}
\author[M. Johnson]{Marianne Johnson}
\address{University of Manchester}
\email{Marianne.Johnson@manchester.ac.uk}
\date{\today}
\begin{document}
\begin{abstract}
Right (and left) coherency 
and right (and left) weak coherency are natural finitary conditions for monoids. Determining whether or not a given monoid has any of these properties is historically a difficult problem.

This paper has several aims, centering around the well-studied class of right (and dually left) $E$-Ehresmann monoids, being one of the broadest classes of monoids containing a semilattice of idempotents. First, we exhibit a particular configuration of elements in a monoid subsemigroup of a right (respectively, left) $E$-Ehresmann monoid, relative to the Ehresmann structure of the overmonoid, that prohibits left (respectively, right) coherence. Second, we apply this technique in a number of different situations. We show that the free Ehresmann monoid of rank at least $2$ is neither left nor right coherent, and that the free left Ehresmann monoid is not left coherent. We demonstrate the utility of our technique in the case where the overmonoid is an $E$-unitary inverse monoid, and apply this to both new situations and to recover the existing results. Namely,   free inverse monoids and free ample monoids of rank at least 2 are neither left nor right coherent, and free left ample monoids of rank at least $2$ is are not left coherent. Next, in a positive direction,  we demonstrate that every free left Ehresmann monoid is weakly coherent.

Our final result is of a different nature. Ehresmann monoids form a variety of monoids with an enriched signature. Viewed as a {\em bi-unary monoid} (respectively, {\em unary} monoid), a free Ehresmann monoid (respectively, free left Ehresmann monoid) does not embed into an inverse monoid. We show that viewed as a {\em monoid} (the standpoint of this paper) every free Ehresmann monoid (and hence also every free left Ehresmann monoid) embeds into an $E$-unitary inverse monoid.
\end{abstract}
\maketitle

\section{Introduction}
An algebra is {\em coherent} if every finitely generated subalgebra has a finite presentation. We say that a monoid $S$ is
{\em right coherent} if every finitely presented right $S$-act is coherent; here an $S$-act is the non-additive version of an $R$-module over a ring. This notion was  motivated by the notion of a {\em coherent theory} \cite{Wh:1976}. It is a {\em finitary condition} (in that every finite monoid is right coherent), and has subsequently been found to have deep connections with  properties of right $S$-acts related to injectivity and  algebraic closure,  and to other finitary properties \cite{DGHRZ20,DG:2023}. A monoid $S$ is {\em weakly right coherent} if,  regarded as a right $S$-act, $S$ is right coherent.  The notions of right coherency and weak right coherency for monoids both correspond  to right coherency for rings; for monoids they are distinct conditions. Coherency is a vibrant topic, which in other contexts has seen recent breakthroughs,  in particular, for groups \cite{JL:2025}.

Understanding which monoids are right, or  dually left, coherent turns out to be hard. Disparate techniques have been needed in each case to show that monoids in certain classes are right coherent, such as the classes of   free (commutative) monoids, of weakly right noetherian regular monoids, or of free left ample monoids \cite{G,GH,DGHRZ20}. Here we focus on providing a tool which quickly {\em rules out} coherency for  monoids belonging to a wide class. 

The classes of right, left and two-sided $E$-Ehresmann monoids are perhaps the broadest classes of monoids possessing a semilattice of idempotents $E$ and for which useful structure results, and descriptions of free algebras, are known \cite{BGG,GG,K,K2}. We show that  a monoid subsemigroup of a right $E$-Ehresmann monoid  containing a certain configuration of elements  is not left coherent. This reflects the approach of a recent article \cite{BGR23}, but there the monoid is required to contain an infinite subgroup. The  configuration here is quite different and has a different range of applications. Most notably, we can apply it to certain submonoids of $E$-unitary (proper) inverse monoids, even if the inverse monoid is aperiodic (group free), as is the case for free inverse monoids.
We use our techniques to show that a free left Ehresmann and a free Ehresmann monoid on a set containing at least two distinct elements are not left coherent (and by a dual argument neither is the free Ehresmann monoid right coherent). On the positive side we demonstrate that the free left Ehresmann monoid is weakly (left and right) coherent.

All $E$-Ehresmann monoids are possessed with two natural unary operations. Viewed as bi-unary monoids, they do not embed into inverse monoids. We prove the  somewhat surprising result that a free  Ehresmann monoid embeds as a monoid into a semidirect product of a group with a semilattice, that is, into an $E$-unitary inverse monoid.

The paper is structured as follows. In Section \ref{sec:tools} we give all the necessary definitions and preliminary results required. In Section \ref{sec:forbid} we provide a survey of the key tools for proving non-coherence in the existing literature,  giving some useful additions (Lemma \ref{lem:m=n} and Corollary \ref{cor:ghe}) to this suite of general results. In Section \ref{sec:proper} we turn our attention to right $E$-Ehresmann monoids, where we present our main result (Theorem \ref{thm:rightEhres}) demonstrating that a submonoid of a right $E$-Ehresmann monoid containing a particular configuration of elements (relative to the right Ehresmann structure of the over monoid) cannot be left coherent. We show that certain conditions of our main result simplify significantly in the case where the right $E$-Ehresmann monoid is an inverse monoid of the form $Y \rtimes G$ where $Y$ is a semilattice and $G$ is a group (Corollary \ref{cor:pi}) and in the case where the right $E$-Ehresmann monoid is $E$-adequate where $E$ is a unitary subset of idempotents (Corollary \ref{cor:E-adequate}). We show (in Remark~\ref{rem:fi}) that this generalises the negative results of \cite{GH}, and exhibit several further natural situations where our result may be invoked, including the Birget-Rhodes expansions (as considered by Szendrei) of certain groups (Remark~\ref{ex:BR}), the Margolis-Meakin expansions of certain groups (Example \ref{ex:MM}), and certain semidirect products of the form $\mathcal{P}(M) \rtimes M$ where $M$ is a monoid acting on its power set by left multiplication (Corollary \ref{ex:freemonoid}). In Section \ref{sec:adequate} we apply our main result (and its left-right dual) to show that the free Ehresmann monoid of rank at least $2$ is neither left nor right coherent and that the free left Ehresmann monoid of rank at least $2$ is not left coherent (Theorem \ref{cor:adequate}). In Section \ref{sec:weakly} we demonstrate that the free left Ehresmann monoid is weakly (left and right) coherent (Theorem \ref{thm:FLEisweaklycoherent}). In Section \ref{sec:embed} we prove that the free Ehresmann monoid embeds (as a monoid) into an inverse monoid of the form $Y \rtimes G$, where $Y$ is a semilattice with identity and $G$ is a group acting on $Y$ by monoid morphisms (Theorem \ref{thm:two-sided}) and demonstrate that this gives an alternative proof (using a different application of our main result) that the free left Ehresmann monoid of rank at least $2$ is not left coherent (Remark \ref{rem:alt}). The paper concludes in Section \ref{sec:questions} with some open questions.

\section{Preliminaries and tools for non-coherence in monoids}
\label{sec:tools}

We provide a concise summary of essential definitions needed for this paper, followed by some of the existing tools used to determine whether a monoid is right or left coherent.  
Throughout the paper,  $S$ denotes a monoid with set of idempotents $E(S)$; unless otherwise stated we denote the identity of $S$ by $1$ or $1_S$.

Many definitions and results that we give have dual  left- and right-handed versions. We do not normally give both explicitly, except where it is necessary to establish notation or for ease of reference. If a monoid has both left- and right-handed versions of a particular property, we drop the adjectives `left' and `right'. By `dual' we always mean the left-right dual. 

For further details concerning monoids and acts, we refer the reader to the monographs \cite{Ho:1995} and \cite{KKM:2000}, respectively.

\subsection{Preliminaries on acts}

\begin{definition}(Right $S$-acts: subacts, morphisms, and generating sets) A right $S$-act is a set $A$ together with a function $A \times S \rightarrow S, (a, s) \mapsto as$, with the property that $(as)t = a(st)$ and $a1= a$ for all $a \in A$ and all $s, t \in S$. For a set $X \subseteq A$ we write $XS := \{xs : x \in X, s \in S\}$. A subset $B \subseteq A$ is a subact of $A$ if $BS \subseteq B$. A morphism between right $S$-acts $A$ and $B$ is a map $\phi : A \rightarrow B$ such that $(as)\phi = (a\phi)s$ for all $a \in A$ and $s \in S$. The $S$-act $A$ is said to be generated by a set $X \subseteq A$ if $XS = A$, finitely generated if it is generated by a finite set, and monogenic if it is generated by a single element.

If the $S$-act $A$ is a monoid (in every case in this paper a monoid semilattice), and if for each $s\in S$ the function $a\mapsto as$ is a monoid morphism, then we say that $S$ acts by monoid morphisms. \end{definition}
The monoid $S$ may be regarded as a right $S$-act over itself and the subacts of $S$ deserve particular mention: a right ideal of $S$ is a subset $I \subseteq S$ such that 
\[IS := \{xs : x \in I, s \in S\} \subseteq I;\] thus,
right ideals of $S$ are precisely the subacts of the (monogenic) right $S$-act $S$.

For a relation $\rho \subseteq A \times A$ on a set $A$, we will use both the notation $(x, y) \in \rho$ and $x \, \rho \, y$ to mean that $x$ and $y$ are $\rho$-related. If $\rho$ is an equivalence relation, we will write $a \rho$ to denote the equivalence class of an element $a \in A$. 

\begin{definition} (Right $S$-act congruences, quotients, generating sets and right $Y$-sequences)
A (right $S$-act) congruence on a right $S$-act $A$ is an equivalence relation $\rho \subseteq A\times A$ with the property that if $a,b \in A$ with $a \, \rho \,  b$ and $s \in S$, then $as \, \rho \, bs$ also holds. 
If $\rho$ is a congruence on a right $S$-act $A$, then the quotient $A/\rho := \{a\rho: a \in A\}$ is a right $S$-act under $(a\rho)s = (as)\rho$, for all $a \in A$ and all $s \in S$. For a set $Y \subseteq A \times A$ we denote by $\langle Y \rangle$ the right congruence on $A$ generated by $Y$, which is the smallest right congruence that contains $Y$. A right congruence $\rho$ is finitely generated if there is a finite set $Y$ such that $\rho = \langle Y \rangle$. For $a, b \in A$ and $Y \subseteq A\times A$ we say that there exists a (right) $Y$-sequence from $a$ to $b$, if there is a sequence of the form
$$a = c_1t_1, d_1t_1 = c_2t_2,\ldots, d_mt_m = b,$$
where $m \geq 0$, $(c_i, d_i) \in Y$ or $(d_i, c_i) \in Y$, $t_i \in S$, for $i = 1, \ldots, m$. We say that $m$ is the length of the sequence, and the case $m = 0$ is interpreted as $a = b$. \end{definition}

The right $S$-act congruences of $S$ considered as a right $S$-act are precisely the right congruences of $S$ considered as a monoid.    We record the following useful facts concerning finitely generated right  congruences.

\begin{lemma}\label{lem:fg}
Let $A$ be an $S$-act, $\rho$ a right (respectively left) congruence on $A$, and $Y \subseteq A$.
\begin{enumerate}[label=\textnormal{(\arabic*)}]
\item The right (respectively left) congruence $\rho$ is
finitely generated if there is no infinite strictly ascending chain of right (respectively left) congruences $\rho_1 \subset \rho_2 \subset \cdots$ such that $\rho =
\bigcup_{n\geq 1} \rho_n$.
\item The element $(a, b) \in A \times A$ is contained in the right (respectively left) congruence generated by $Y$ if and only if there exists a right (respectively left) $Y$-sequence from $a$ to $b$.
\end{enumerate}
\end{lemma}

\begin{definition}(Free and finitely presented right $S$-acts)
Since right $S$-acts form a variety of algebras, the free right $S$-act $F_S(X)$ over any set $X$ exists; concretely, $F_S(X) = X \times S$ with the action of $S$ given by $(x, s) \cdot t = (x, st)$. A right $S$-act $A$ is finitely presented if it is isomorphic to $F_S(X)/\rho$ for some finite set $X$ and finitely generated congruence $\rho$.
\end{definition}

\begin{definition}(Right coherence and weak right coherence)
\label{sec:weak}
A monoid $S$ is right coherent if every finitely generated subact of every finitely presented right  $S$-act is finitely presented. A monoid $S$ is weakly right coherent if every finitely generated right ideal of S is finitely presented as a right $S$-act.
\end{definition}
Clearly every right coherent monoid is in particular weakly right coherent.

\subsection{Preliminaries on monoids}

\begin{definition}\label{defn:green} (Green's relations $\mathcal{L},\, \mathcal{R}$) 
    For a monoid $S$, Green’s preorder $\leq_{\mathcal{R}}$ on $S$ is defined by $a \leq_{\mathcal{R}} b$ if $aS \subseteq bS$ and Green’s equivalence $\mathcal{R}$ is the equivalence relation induced by the preorder $\leq_\mathcal{R}$, that is, $a\,  \mathcal{R}\,  b$ if and only if $aS = bS$.  The  relation $\mathcal{R}$  is a left congruence. The dual relations to   $\leq_\mathcal{R}$ and  $\mathcal{R}$ are denoted by  $\leq_{\mathcal{L}} $ and $\mathcal{L}$, respectively.
    \end{definition}

    \begin{definition}\label{defn:reg} (Regular, inverse) 
    An element $x \in S$ is said to be \emph{regular} if there exists $y \in S$ such that $xyx = x$;  then $S$ is \emph{regular } if every $x \in S$ is regular. An equivalent definition of regularity of $S$ is that every $\mathcal{L}$-class (dually, every $\mathcal{R}$-class) contains an idempotent. A regular element $x \in S$ has an \emph{inverse} $z$, that is, $z$ satisfies $xzx = x$ and $zxz = z$. If every element has a unique inverse, then $S$ is said to be  \emph{inverse } and in this case we write $x^{-1}$ to denote the unique inverse of $x$. Equivalently, a semigroup  is inverse if it is regular and the idempotents form a semilattice, that is, a commutative semigroup of idempotents. The term semilattice is used since the idempotents are then partially ordered under $e\leq f$ if $ef=e$, with $ef$ being the greatest lower bound of any $e,f\in E(S)$.
\end{definition}

Inverse monoids feature in several ways in this work by providing  important oversemigroups within which we may embed monoids from wider classes. All of those we consider are in fact $E$-unitary (otherwise known as proper).

\begin{definition}\label{defn:eunitary}(Right unitary, $E$-unitary)  A subset $F\subseteq  E(S)$ of a monoid  $S$ is right 
unitary if $e,ae\in F$ implies that $a\in F$. We note that if $F=E(S)$ this is equivalent to the dual condition and in this case we say that $S$ is $E$-unitary.
    \end{definition}

We occasionally make use of the notion of being $E$-unitary per se, but more often we are interested in $E$-unitary inverse monoids, which relate to semidirect products via a  key result of O'Carroll, given as Theorem~\ref{thm:pi} below.

\begin{definition} (Semidirect product)  Let $S$ be a monoid acting on the left of a monoid semilattice $(Y, \wedge)$ by monoid morphisms. Then $Y\rtimes S:= Y\times S$ is a monoid  under 
$(y,s)(z,t)=(y\wedge s z, st)$ with identity $(1_Y,1_S)$.
\end{definition}

\begin{thm}\label{thm:pi}\cite{Oc:1976}  An inverse monoid is $E$-unitary if and only if it embeds (as an inverse semigroup) into a semidirect product $Y\rtimes G$, where $Y$ is a semilattice with identity and $G$ is a group acting by monoid morphisms. In this case, in $Y \rtimes G$  we have  $(y,g)^{-1} = (g^{-1}y, g^{-1})$, 
  $(y,g)(y,g)^{-1}=(y,1_G)$,
   $(y,g)^{-1}(y,g)=(g^{-1}y,1_G)$
  and set of idempotents $E = \{(y, 1_G): y \in Y\}$.  \end{thm}

\begin{definition}(Right annihilator congruences, the relations $\mathcal{L}^*$ and $\mathcal{R}^*$,  right abundance)  For $a \in S$ and a right congruence $\rho$ on $S$, let
$$\mathbf{r}(a\rho) := \{(u, v) \in  S \times S : au 
\, \rho \, av\}.$$
It is straightforward to check that $\mathbf{r}(a\rho)$ is a right congruence; we call this the `right annihilator congruence of $a$ with respect to $\rho$'. In the case where $\rho$ is equality, we will simply write $\mathbf{r}(a)$ and refer to this as `the right annihilator congruence of $a$'. We denote by $\mathcal{L}^*$ the equivalence relation on $S$ defined (for all $a,b \in S$) by $a \,\mathcal{L}^*\, b$ if $\mathbf{r}(a) = \mathbf{r}(b)$. The relation $\mathcal{L}^*$ is an extension  of Green's $\mathcal{L}$-relation, in the sense that $\mathcal{L}\subseteq \mathcal{L}^*$, and is a right congruence. We say that $S$ is right abundant if every $\mathcal{L}^*$-class of $S$ contains an idempotent. If $e \in S$ is idempotent, then $(e,1) \in \mathbf{r}(e)$ and hence $ae=a$ for all $ a\in S$ with $a \, \mathcal{L}^* \, e$.
The 
dual of $\mathcal{L}^*$ is denoted by $\mathcal{R}^*$.
\end{definition}

If the set of idempotents $E(S)$ of any monoid $S$ is a semilattice, then (as for $\mathcal{L}$) it is not hard to see that each $\mathcal{L}^*$-class a contains at most one idempotent.

We present the first of several extensions of the class of inverse monoids. In each case, we retain a semilattice of idempotents, but weaken the condition of regularity.

\begin{definition} (Right ample monoids)  A monoid $S$ is said to be \emph{right ample} if the idempotents form a semilattice, for each $a \in S$ the $\mathcal{L}^*$-class of $a$ contains a unique idempotent $a^*$ and the right ample identity  $ea = a(ea)^*$  holds for all $a \in S$ and $ e \in E(S)$.  Dually, $S$ is said to be \emph{left ample} if the idempotents form a semilattice, for each $a \in S$ the $\mathcal{R}^*$-class of $a$ contains a unique idempotent $a^+$ and the left  ample identity $ae = (ae)^+a$ holds for all $a \in S$ and all $ e \in E(S)$. A semigroup  is ample if it is both left and right ample.
\end{definition}

Right ample monoids may also be defined as quasi-varieties of unary monoids (that is, monoids equipped with an additional basic unary operation) and dually for left ample monoids. Similarly ample monoids may  be defined as quasi-varieties of  bi-unary monoids (that is, monoids equipped with two additional basic unary operations). In this article it is convenient to treat right (left, two-sided) ample monoids as  monoids.

If $S$ is inverse, then $S$ is ample with $\mathcal{R}=\mathcal{L}^*$,\, $\mathcal{L}=\mathcal{L}^*$, 
$a^*=a^{-1}a$ and $a^+=aa^{-1}$.

\subsection{Necessary and sufficient conditions for (weak) right coherency}
In this subsection we focus on stating results for (weak) right coherence. In each case there is a corresponding left-right dual result for (weak) left coherence.   
\begin{theorem}
\cite[Corollary 3.4]{G}
\label{thm:coherent}
Let $S$ be a monoid. The following are equivalent:
\begin{enumerate}[label=\textnormal{(\arabic*)}]
\item $S$ is right coherent;
\item for any finitely generated right congruence $\rho$ on $S$ and any elements $a, b \in S$:
\begin{enumerate}[label=\textit{(\roman*)}]
    \item the subact $(a\rho)S \cap (b\rho)S$ of the right $S$-act $S/\rho$ is finitely generated;
\item
$\mathbf{r}(a\rho)$ is a finitely generated right congruence on $S$.
\end{enumerate}
\end{enumerate}
\end{theorem}

\medskip
In the case where $S$ is regular, there is a sufficient condition for right  coherency,  phrased simply in terms of the right ideal structure of $S$. For a right  ideal $I$ of $S$, $\rho$ a right congruence on $S$ and $x \in I$, let us write $I\rho :=\{s \in S: s \,\rho\, a \mbox{ for some } a \in I\}$ and $[I, x] :=\{s \in S: xs \in I\}$; it is easy to see that these are also right ideals of $S$.

\begin{proposition}\cite[Theorem 3.2]{DGHRZ20}
\label{prop:regcoherent}
Let $S$ be a regular monoid. If for every finitely generated right congruence $\rho$ and all $a, b, x, y \in S$  the right ideals  $(aS)\rho \cap (bS)\rho$ and  \\ $[aS, x] \cap [bS, y]$ are finitely generated, then $S$ is right coherent.
\end{proposition}

\begin{theorem}
\cite[Corollary 3.3]{G}
\label{thm:weakcoherent}
Let $S$ be a monoid. The following are equivalent:
\begin{enumerate}[label=\textnormal{(\arabic*)}]
\item $S$ is weakly right coherent;
\item for all $a, b \in S$:
\begin{enumerate}
    \item[(i)] $aS \cap bS$ is finitely generated;
\item[(ii)]
$\mathbf{r}(a)$ is a finitely generated right congruence on $S$.
\end{enumerate}
\end{enumerate}
\end{theorem}

It will be convenient to have some terminology corresponding to the conditions of the previous theorem.
\begin{definition}(Right ideal Howson and finitely right equated monoids) A monoid with the property that the intersection of any two principal (or equivalently two finitely generated) right ideals is finitely generated is said to be {\em right ideal Howson}. A monoid $S$ with the property that each right annihilator congruence of the form $\mathbf{r}(a)$ where $a \in S$ is finitely generated is said to be {\em finitely right equated}. 
\end{definition}
We note that the term `finitely right aligned' used elsewhere in the literature coincides with the property of being right ideal Howson for \emph{monoids}. (Right ideal Howson \emph{semigroups} need not be finitely aligned, however \cite{CG21}.)

For regular monoids, each right annihilator $\mathbf{r}(a)$ is generated by $(1,ba)$, where $b$ is any inverse of $a$, and so the notions of weak right coherence and right ideal Howson coincide.

\begin{proposition}\cite[Corollary 3.3]{BGR23}
\label{prop:regweak}
Let $S$ be a regular monoid.
Then $S$ is weakly right coherent if and only if $aS \cap bS$
is finitely generated for all $a, b \in S$.
\end{proposition}

\begin{remark}
\label{rem:inverse}
Note that Proposition \ref{prop:regweak} in particular gives that every inverse monoid $S$  is both weakly right coherent, since
$aS\cap bS=aa^{-1}bb^{-1}S$, and dually weakly left coherent. More generally, if $S$ possesses an involution $\circ$ satisfying $(ab)^\circ = b^{\circ}a^{\circ}$, then $S$ is (weakly) right  coherent if and only if $S$ is (weakly) left  coherent.
\end{remark}

In a similar vein, using Theorem \ref{thm:coherent} (\cite[Corollary 3.4]{G}) we have the following:
\begin{proposition}\cite[Lemma 3.5]{G}
\label{prop:rabundweak}
Let $S$ be a right abundant monoid. Then $S$ is finitely right  equated. Consequently, $S$ is weakly right  coherent if and only if $S$ is right  ideal Howson.
\end{proposition}

Another useful technique is that coherence behaves well with respect to retraction. Recall that a submonoid $S$ of a monoid $M$ is a {\em retract} of $M$ if there exists a surjective monoid homomorphism $\varphi: M \rightarrow S$ such that $\varphi^2 = \varphi$.

\begin{theorem}\cite[Corollary 3.5, Corollary 4.12, Theorem 5.5]{DGM24} and \cite[Theorem 6.2]{GH}
\label{thm:retract} The class of right ideal Howson (finitely right equated, weakly right coherent,  right coherent) monoids is closed under retract.
\end{theorem} 

\section{Forbidden configurations}
\label{sec:forbid}
Other than the result concerning retracts mentioned above, there is not a particularly `neat' relationship between coherency and more general submonoids. Clearly, every non-coherent monoid will contain coherent submonoids (e.g. the trivial submonoid $\{1\}$ is finite and so, in particular, coherent). On the other hand, a coherent monoid can contain non-coherent submonoids (e.g. if $|X|\geq 3$, then it has been shown in \cite[Example 6.2]{DGHRZ20} that the monoid formed as the direct product of two copies of the free monoid over $X$, that is, $X^* \times X^*$, is not 
coherent, however, this monoid clearly embeds into a group, and by \cite{G} all groups are coherent as monoids). Nevertheless, it can be shown that certain configurations of elements are a barrier to left (or right) coherence. The development of such configurations was initiated in \cite{BGR23}.

Throughout the paper we use the following notations and definitions.

\begin{definition}(Special  annihilators)
Let $S$ be a monoid and  $a, b \in S$  and let $l_a$ (and  $r_a$) denote the left  (and respectively right) congruence on $S$ generated by the pair $(a, 1)$. Consider the left congruence  
$$\lambda_{a, b} =\mathbf{l}(b l_a) =   \{(u,v) \in S \times S: u b \, l_a \, v b\} $$ and dually the right congruence 
$$\rho_{a, b} = \mathbf{r}(b r_a) = \{(u,v) \in S \times S: b u  \, r_a \,  b v \}.$$
\end{definition} These congruences are particularly nice to work with, due to the following:

\begin{lemma}(See for example \cite[Lemma 7.1 and Lemma 7.2]{GH} and \cite[Lemma 3.1]{BGR23}.)\\
\label{lem:powers}
In the notation above, we have
\begin{eqnarray*}
l_a = \{(u,v): ua^m= va^n, \mbox{ for some } m, n \geq 0\},\,\, \lambda_{a, b} = \{(u,v): uba^m= vba^n, \mbox{ for some } m, n \geq 0\}.
\end{eqnarray*}
\end{lemma}
It follows immediately from Theorem \ref{thm:coherent} that if there exist $a,b \in S$ such that  $\lambda_{a, b}$ is \emph{not} finitely generated, then we may conclude that $S$ is not  left coherent. The following observation demonstrates a certain `forbidden configuration' of elements for left coherent monoids.

\begin{lemma}
\label{lem:m=n}
Let $M$ be a monoid containing elements $a$ and $b$ such that
\begin{enumerate}[label=\textnormal{(\arabic*)}]
    \item for all $u,v \in M$, if $uba^n = vba^m$ holds for some $m,n \geq 0$ then $uba^n = vba^n$ also holds;
    \item for each $i \geq 1$ there exist $u_i,v_i \in M$ such that $u_i ba^i = v_iba^i$, but $u_i ba^{i-1} \neq v_iba^{i-1}$.
\end{enumerate}
Then $M$ is not left coherent.
\end{lemma}
\begin{proof} With 
$l_{a}$ and $\lambda_{a,b}$ defined as above, by Lemma \ref{lem:powers} and condition (1) we have that \begin{eqnarray*}
 \lambda_{a,b} &=& \{(u,v) \in M \times M: 
uba^n= vba^m, \mbox{ for some } m, n \geq 0\},\\
&=& \{(u,v): uba^i= vba^i, \mbox{ for some }i \geq 0\},\\
&=& \bigcup_{i \geq 0} \lambda_i \mbox{ where }\lambda_i = \{(u,v): uba^i = vba^i\}.
\end{eqnarray*}
Clearly $\lambda_i=l(ba^i)$ is a left congruence and $\lambda_{i-1} \subseteq \lambda_i$ for all $i \geq 1$.  By condition (2) we see that $(u_i, v_i) \in \lambda_i \setminus \lambda_{i-1}$,  and then Lemma~\ref{lem:fg} gives that the congruence $\lambda_{a,b}$ on $M$ is not finitely generated.    
\end{proof}

Recall that for non-empty set $X$ and $R$ a relation on $X^*$, the monoid presentation $\langle X \mid  R \rangle$ defines the quotient monoid $X^*/R^\#$, where $R^\#$ denotes the two-sided congruence generated by $R$. If $X$ and $R$ are countable sets, say $X = \{x_0,x_1, x_2,\ldots\}$ and $R = \{(u_1, v_1), (u_2, v_2),\ldots\}$, it is customary to write simply $\langle x_i,\,  i \geq 0\mid u_i=v_i,\, i \geq 0 \rangle$.

\begin{example} Consider the monoid $M$ with presentation
\[M = \langle a,b,u_i,v_i,\, i\geq 0 \mid  u_iba^i=v_iba^i, i\geq 0\rangle.\]
Suppose that  $uba^n=vba^m$ for some $u,v\in M$; as all the relations preserve the number of $a$'s following the last $b$ in a word,  it follows that that $m=n$.  Then it is easy to see that $u_iba^{i-1}\neq v_iba^{i-1}$ for all $i\geq 0$ and so $M$ is not left coherent.
 Along the same lines, \[N = \langle a,b,c,d,u_i,v_i,s_i,t_i,\, i\geq 0 \mid  u_iba^i=v_iba^i, \, c^ids_i=c^idt_i,\,  i\geq 0\rangle\]
   is neither left nor right coherent.
    \end{example}
    
The approach of looking for forbidden configurations was introduced in \cite{BGR23} to prove the following  result, which transpires to have many applications.

\begin{theorem} \label{thm:forbid}\cite[Theorem 4.3]{BGR23}
\label{BGR}
Let $S$ be a monoid, and suppose that $g,h,e \in S$ are elements satisfying all the following conditions.
\begin{enumerate}[label=\textnormal{(\arabic*)}]
\setcounter{enumi}{-1}
    \item $e^2=e$, $gh=hg$, $ghg=g$, $hgh=h$.
    \item $hge=e=ehg$.
    \item For all $n>0$, $eg^neh^n = g^neh^ne$ and $eh^neg^n = h^neg^ne$.
    \item \begin{enumerate}[label=\textit{(\roman*)}]
        \item For all $m,n>0$, $g^meh^m \neq g^meh^m h^neg^n \neq  h^neg^n$.
        \item For all $m\neq n \geq 0$, $g^meh^m \neq g^meh^m g^neh^n$ and $h^meg^m \neq h^meg^m h^neg^n$.
    \end{enumerate}
    \item \begin{enumerate}[label=\textit{(\roman*)}]
        \item For all $0<k<n$, $eg^neh^n \neq eg^neh^ng^keh^k$. 
        \item For all $0<k\leq n$, $eg^neh^n \neq eg^neh^nh^keg^k$. 
    \end{enumerate}
\end{enumerate}
Then $S$ is neither right
nor left coherent.
\end{theorem}

In \cite{BGR23} it is shown that if $S$ is any monoid satisfying the conditions of \cite[Theorem 4.3]{BGR23}, then taking $a = g$ and $b = e$ will yield that $\lambda_{a, b}$ and $\rho_{a, b}$ are not finitely generated. Any monoid satisfying the conditions of Theorem \ref{BGR} must be a quotient of a special semidirect product, as we shall now explain.

\begin{definition} (Special semidirect products) 
Let $M$ be a monoid and denote by $\mathcal{P}(M)$  the power set of $M$ viewed
as a monoid (indeed, a semilattice), with identity  $\emptyset$, with respect to the operation of union. It is clear that $M$ acts by monoid morphisms on the left of $\mathcal{P}(M)$ via $( m,X) \mapsto mX = \{mx: x \in X\}$ for all $X \in \mathcal{P}(M)$ and all $m \in M$. The semidirect product 
$\mathcal{P}(M) \rtimes M$ with respect to this action is then the monoid 
whose elements lie in $\mathcal{P}(M) \times M$,
with product given by $(X,x)(Y,y) = (X \cup xY, xy)$, and with identity element $(\emptyset, 1)$, where $1$ is the identity element of $M$. Writing $\mathcal{P}^{\rm fin}(M)$ to denote the subset of $\mathcal{P}(M)$ whose elements are finite subsets of $M$,
it is then easy to see that  $\mathcal{P}^{\rm fin}(M) \rtimes M$ is a submonoid of $\mathcal{P}(M) \rtimes M$. Throughout the paper we will write $\mathcal{S}(M) = \mathcal{P}(M) \rtimes M$ and $\mathcal{S}^{\rm fin}(M) = \mathcal{P}^{\rm fin}(M) \rtimes M$ and refer to them as  special semidirect products.
\end{definition}

We will see  that special semidirect products form a very broad class of monoids. In the case where $M$ is  a group, 
$\mathcal{S}(M)$ is inverse, and such monoids contain many important $E$-unitary inverse monoids.

\begin{remark}
Let $P(g,h,e)$ be the semigroup considered in \cite{BGR23} with presentation $P(g,h,e):=\langle g,h,e\mid  R\rangle$ where $R$ is the set of relations given by conditions (0), (1) and (2) of Theorem \ref{BGR}. As remarked in \cite{BGR23}, the monoid  $P(g,h,e)$ is isomorphic to the semidirect product $\mathcal{S}^{\rm fin}(\mathbb{Z})$ with the integers here considered as a group under addition. We now complete  the argument, following on from \cite[Proposition 4.1]{BGR23}. From the relations in conditions (0) and (1) one sees that: $P(g,h,e)$ is a monoid with identity element $gh$; and $g$ lies in the group of units of $P(g,h,e)$ with inverse $h$. Writing $g^{-1}=h$ and $g^0=gh$ each element of $P(g,h,e)$ is then equal to a product of the form $g^{k_1}eg^{k_2}e \cdots g^{k_{n-1}}eg^{k_{n}}$ where $n \geq 1$, and $k_1, \ldots, k_n \in \mathbb{Z}$. The elements of the form $g^jeg^{-j}$ where $j \in \mathbb{Z}$ are all distinct pairwise commuting (using relations from (2)) idempotents. For each finite subset $J$ of the integers, let us denote by $e_J$ the product (in any order) of the idempotents $g^jeg^{-j}$ for all $j \in J$ (where if $J = \emptyset$, we set $e_J =g^0$). Distinct subsets of $\mathbb{Z}$ yield distinct idempotents. It is readily verified that each  $\gamma\in  P(g,h,e)$ can be uniquely rewritten as $\gamma = e_J g^k$ for some finite subset $J$ of the integers and some integer $k$. The product of two such elements $e_{J_1}g^{k_1}$ and $e_{J_2}g^{k_2}$ (where $J_1, J_2$ are finite subsets of $\mathbb{Z}$ and $k_1, k_2 \in \mathbb{Z}$) can be expressed as follows:
\begin{equation}
 \label{prod}
e_{J_1}g^{k_1} e_{J_2}g^{k_2} = e_{J_1 \cup (k_1+J_2)} g^{k_1+k_2}.
 \end{equation}
To see this: note that for all integers $s,k$ we have
$$g^{k} (g^s e g^{-s}) = (g^{k+s} e g^{-s}) g^{k-k} = (g^{k+s} e g^{-(k+s)}) g^{k}.$$ Thus the $g^{k_1}$ in the middle of the left-hand side of \eqref{prod} can be pushed past each idempotent $g^seg^{-s}$ to the right of it at the expense of adding $k_1$ to the exponent $s$. The formula \eqref{prod} then follows by induction. Due to the uniqueness of expression of elements of $P(g,h,e)$ in the form $e_Jg^k$, it is then clear that there is a well-defined  map $\Psi: P(g,h,e) \rightarrow \mathcal{S}^{\rm fin}(\mathbb{Z})$ given by $\Psi(e_Jg^k) = (J,k)$ and moreover, it follows easily from equation \eqref{prod} that this is a monoid isomorphism. Notice that under this map we have $g\Psi = (\emptyset, 1), h\Psi = (\emptyset, -1)$ and $e\Psi =  (\{0\}, 0)$, respectively. It then follows from \cite[Proposition 4.2]{BGR23}, that $g =(\emptyset, 1)$, $h=(\emptyset, -1)$, $e =(\{0\},0)$ are elements of $\mathcal{S}(\mathbb{Z})$ satisfying the conditions of \cite[Theorem 4.3]{BGR23}, and hence taking $\alpha = g$ and $\beta = e$ yields that $\lambda_{\alpha, \beta}$ and $\rho_{\alpha, \beta}$ are not finitely generated.
\end{remark}

With this point of view in mind, we give the following rephrasing of Theorem \ref{BGR}:
\begin{corollary}
\label{cor:ghe}
Let $M$ be a monoid containing a subgroup $\{x^k: k \in \mathbb{Z}\}$ isomorphic to $\mathbb{Z}$. Suppose that $\rho$ is a congruence on $\mathcal{S}(M)$ satisfying the following conditions.
\begin{enumerate}[label=\textnormal{(\arabic*)}]
\item For all $m, n \in \mathbb{Z}$ with $m \neq n$, $(\{x^m\}, x^0)$ is not $\rho$-related to $(\{x^n\}, x^0)$. 
    \item For all $m,n>0$, $(\{x^m,x^{-n}\}, x^0)$ is not $\rho$ related to  $(\{x^m\}, x^0)$ or $(\{x^{-n}\}, x^0)$. 
    \item For all $m\neq n\geq 0$, $(\{x^m,x^n\}, 
    x^0)$ is not $\rho$ related to  $(\{x^m\}, x^0)$ and $(\{x^{-m},x^{-n}\}, x^0)$ is not $\rho$ related to  $(\{x^{-m}\}, x^0)$.
    \item For all $0<k<n$, $(\{x^0,x^n\}, x^0)$ is not $\rho$ related to $(\{x^k\}, x^0)$  or to $(\{x^0,x^k,x^n\}, x^0)$.
    \item For all $0<k \leq n$, $(\{x^0,x^n\}, x^0)$ is not related to $(\{x^{-k}\}, x^0)$ or to $(\{x^{-k},x^0,x^n\}, x^0)$.
\end{enumerate}
Then any monoid $S$ containing the quotient $Q=\mathcal{S}(M)/\rho$ as a subsemigroup is not right or left coherent.
\end{corollary}
\begin{proof} We have observed, in different notation that  $g =(\emptyset, x)$, $h=(\emptyset, x^{-1})$ and $e =(\{x^0\},x^0)$ are elements of $\mathcal{S}(M)$ satisfying the conditions of \cite[Theorem 4.3]{BGR23}.  Now  if $\rho$ satisfies conditions (1) and (2) above, then the corresponding inequalities of condition (3)($i$) of the Theorem will still hold in the quotient $Q$. Likewise, conditions (3), (4) and (5) will ensure that conditions (3)($ii$), (4)($i$) and (4)($ii$) also hold in $Q$.
\end{proof}

\begin{example}For a group $G$ and a natural number $n\geq 1$ let 
$$Q_n(G)=\{(X,x): X \subseteq G, |X| < n \mbox{ or } X = G, x \in G\},$$
with product
$$(X,x)(Y,y) = \begin{cases}
(X \cup xY, xy) &\mbox{ if }|X \cup xY| <n\\
(G, xy) &\mbox{ else}.
\end{cases} 
$$
Notice that $Q_n(G)$ is the quotient of $\mathcal{S}(G)$ by the congruence $\rho_n$ defined by $(X,x)\, \rho_n\, (Y,y)$ if and only if $x=y$ and either (i) $X=Y$ or (ii) $|X|\geq n$ and $
|Y| \geq n$. It is clear that that each such monoid is inverse; if $(X, x) \in Q_n(G)$, then $(x^{-1}X, x^{-1}) \in Q_n(G)$ is the unique generalised inverse of $(X,x)$.  Then by Remark~\ref{rem:inverse} we have that $Q_n(G)$ is weakly coherent. For $n=1$ it is easy to see that monoid $Q_1(G) =\{(\emptyset, x), (G, x): x \in G\}$ is  isomorphic to the direct product $\mathcal{P}(\{1\}) \times G$ and has the property that each right ideal is finitely generated, and so it follows from \cite[Corollary 3.3]{DGHRZ20} together with the fact that $Q_1(G)$ is inverse  that $Q_1(G)$ is coherent. On the other hand, if $G$ is infinite then for each $n \geq 2$ the monoid $Q_n(G)$ contains non-finitely generated right ideals (e.g. for any subset $X$ of $G$, the right ideal generated by the elements $\{(\{x\}, 1): x \in X\}$ has no fewer than $|X|$ generators). Moreover, if $G$ contains an element of infinite order, then by Corollary \ref{cor:ghe} any monoid containing $Q_n(G)$ as a subsemigroup for some $n\geq 3$ is neither left or right coherent. \end{example}

In \cite[Proposition 4.1 and Proposition 4.2]{BGR23}, it is observed that the maps $g: x \mapsto x+1$ for all $ x\in \mathbb{Z}$, $h: x \mapsto x-1$ for all $ x\in \mathbb{Z}$, and $e: x \mapsto x$ for all $x \neq 0$ are elements of the symmetric inverse monoid $\mathcal{I}_\mathbb{Z}$ satisfying the conditions of the theorem. Thus $\mathcal{I}_{\mathbb{Z}}$, or indeed, any monoid containing as a subsemigroup a monoid isomorphic to the submonoid generated by these three maps, such as: the full transformation monoid $\mathcal{T}_X$, partial transformation monoid $\mathcal{PT}_X$, symmetric inverse monoid $\mathcal{I}_X$, or partition monoid $\mathcal{P}_X$ on any infinite set $X$ \cite[Theorem 5.1]{BGR23} will be neither left nor right coherent. Notice that since $g,h,e$ are order preserving and injective, it also follows that the monoid of order preserving injective partial transformations on $\mathbb{Z}$ is neither left nor right coherent. Similarly, the triple $g: x \mapsto x+1$ for all $ x\in \mathbb{Z}$, $h: x \mapsto x-1$ for all $ x\in \mathbb{Z}$, and $e: x \mapsto x$ for all $x \neq 0$ and $e: 0 \mapsto 1$ is easily seen to satisfy the conditions of the theorem, hence showing that the monoid of all order preserving (full) transformations on $\mathbb{Z}$ is neither left nor right coherent.

\medskip
In separate work \cite{GH}, the free inverse monoid, free ample monoid, and free left ample monoid on at least two generators have each been shown not to be left coherent (with the free inverse monoid and free ample monoid also not right coherent). The technical details of these proofs are overall quite different to those from \cite{BGR23}. A careful analysis of the proof reveals it hinges on elements with certain particular properties; one of our aims subsequently is to extract such properties to develop a forbidden configuration.  Moreover, elements of special semidirect products  $\mathcal{S}(G)$ for a group $G$ also make an appearance. Indeed, let $F_X$ be the free group on a set $X$ with identity element $1$ and recall from \cite{S} that one can view the free inverse monoid $\FI(X)$, the free ample monoid $\FA(X)$, and the free left ample monoid $\FLA(X)$ on $X$ as monoid subsemigroups of $\mathcal{S}^{\rm fin}(F_X)$ as follows:
\begin{eqnarray*}
\FI(X)&:=& \{(A,a): A \in \mathcal{P}^{\rm fin}(F_X) \mbox{ is non-empty and prefix closed},  a \in A\},\\
\FA(X)&:=& \{(A,a) \in \FI(X): a \in X^*\} \subseteq \FI(X),\\
\FLA(X)&:=& \{(A,a) \in \FI(X): A  \subseteq X^*\} \subseteq \FA(X).
\end{eqnarray*} Notice that  elements $(A,a)$ of each of these monoids satisfy (i) the set $A$ is non-empty and prefix closed, and (ii) $a \in A$, and so these subsemigroups do not contain the identity element $(\emptyset, 1)$ of $\mathcal{S}^{\rm fin}(F_X)$ and moreover we are also immediately prevented from working with special annihilators using the elements  $(\emptyset, x)$ and $ (\{1\},1)$, as in \cite{BGR23}. The proof given in \cite{GH} instead proceeds by taking the elements $ (\{1, x\}, x)$ and $(\{1, y\}, y)$, where $x, y$ are distinct elements of $X$, and relies heavily on the fact that every element $(A,a)$ of the monoids in question has the property that $A$ is prefix-closed and contains $a$.

\medskip
The results discussed above cover the large majority of known examples of non-coherent monoids.

\section{Forbidden configurations in right {\it E}-Ehresmann monoids}
\label{sec:proper}

Right $E$-Ehresmann monoids form a broad class of monoids including several well-known classes we have mentioned above,  such as inverse monoids. 
In this section we  consider forbidden configurations for coherency in such monoids, arriving at a result with wide applicability. We give a brief introduction to Ehresmann monoids below; 
for more information, we refer the reader to \cite{Lawson:91,BGG}. 

\begin{definition}\label{defn:els} (The relations $\leq_{\widetilde{\mathcal{L}}_E}$, $\leq_{\widetilde{\mathcal{R}}_E}$, $\widetilde{\mathcal{L}}_E$ and
$\widetilde{\mathcal{R}}_E$) Let $S$ be a monoid and let $E \subseteq E(S)$. We recall that  $\leq_{\widetilde{\mathcal{L}}_E}$ denotes  the pre-order on $S$ defined by the rule that for all $a,b \in S$ we have that 
\[a\leq_{\widetilde{\mathcal{L}}_E} b\mbox{ if and only if } \{ e\in S: be=b\} \subseteq \{ e\in S: ae=a\}.\] The associated equivalence relation is denoted by $\widetilde{\mathcal{L}}_E$. 
Thus, $a\,\widetilde{\mathcal{L}}_E\, b$
if and only if $a$ and $b$ have the same right identities from $E$. The relations $\leq_{\widetilde{\mathcal{R}}_E}$ and $\widetilde{\mathcal{R}}_E$ are defined dually.\end{definition}

The relation $\widetilde{\mathcal{L}}_E$ is a generalisation of both Green's $\mathcal{L}$-relation and the relation $\mathcal{L}^*$: indeed, we have that $\mathcal{L} \subseteq \mathcal{L}^* \subseteq \widetilde{\mathcal{L}}_E$, and  if $S$ is regular we have $\mathcal{L} = \widetilde{\mathcal{L}}_{E(S)}$. In general, however, unlike the relations $\mathcal{L}$ and $\mathcal{L}^*$, the relation $\widetilde{\mathcal{L}}_E$ need not be a right congruence. In the case where $E$ is a semilattice it is however easy to see that each $\widetilde{\mathcal{L}}_E$-class contains at most one idempotent of $E$.
\begin{definition}\label{def:ehresmann} (Right $E$-Ehresmann) Let $E\subseteq E(S)$ be a semilattice.  We say that $S$ is {\em right $E$-Ehresmann} if every $\widetilde{\mathcal{L}}_E$-class of $S$ contains exactly one element of $E$ and $\tilde{\mathcal{L}}_E$ is a right congruence. In the case where $E=E(S)$ we abbreviate right $E$-Ehresmann to right Ehresmann. We  denote by $a^*$ the idempotent of $E$ in the $\widetilde{\mathcal{L}}_E$-class of $a$. (Dually,  in a left $E$-Ehresmann monoid, we denote by $a^+$ the idempotent of $E$ in the $\widetilde{\mathcal{R}}_E$-class of $a$.)
    \end{definition}
    
If $E(S)$ is a semilattice,  then since each $\widetilde{\mathcal{L}}_{E(S)}$-class contains at most one idempotent and $\mathcal{L} \subseteq \mathcal{L}^* \subseteq \widetilde{\mathcal{L}}_{E(S)}$, it immediately follows that if $S$ is inverse or right ample, then $S$ is right Ehresmann. In the case where $S$ is inverse, $a^+ = aa^{-1}$ and $a^* = a^{-1}a$.

There is another approach to right $E$-Ehresmann monoids that we should mention, which corresponds to an earlier remark made in the context of right ample monoids. In a right $E$-Ehresmann monoid  there is a unary operation determined by $a \mapsto a^*$. Right $E$-Ehresmann monoids then form a variety of unary monoids. If regarded in this way we may refer to $E$ as the {\em semilattice of projections.}  The present approach, regarding right $E$-Ehresmann monoids as plain monoids rather than unary monoids, is the most useful for potential applications.

Our first main result, Theorem \ref{thm:rightEhres}, gives an obstruction to  left coherence for certain monoid subsemigroups  of a right $E$-Ehresmann monoid $F$. Whilst the conditions of this theorem  are technical, we shall see (in this and subsequent sections) that there are many natural circumstances in which they are satisfied, and in which they simplify. This includes in particular  the case where $F$ is taken to be a free inverse monoid. Our result can be applied to some interesting submonoids of the free inverse monoid, in particular, recovering the negative results from \cite[Theorem 7.5]{GH}.

\begin{thm}\label{thm:rightEhres} Let $F$ be a right $E$-Ehresmann monoid and let $M$ be a monoid subsemigroup of $F$ such that:
\begin{enumerate}[label=\textnormal{(\arabic*)}]
\item there exists a unary operation $t \mapsto t^+ \in E$ on $F$ such that if $e=tx=e^2 \in E$, where $t,x\in M$ then $e=t^+x^*$;
\item if $tx\in E$ and  $xw=yw$ where $t,x,y,w \in M$, then $ty\in E$;
\item there exist elements $a,b\in M$ such that 
\begin{enumerate}[label=\textit{(\roman*)}]
\item $b,ba, ba^2,\ldots$ are all pairwise incomparable under the $\leq_{\widetilde{\mathcal{L}}_E}$-order of $F$,
\item for each $i\geq 0$, we have $M \cap (E_{i+1}
 \setminus E_{i}) \neq \emptyset$ where $E_i :=\{f \in E: fba^i=ba^i\}$,
 \item for each $i \geq 0$ and $x,y \in M$, if $xba^i = yba^i$ and $x^*ba^i=ba^i$, then $y^*ba^i=ba^i$.
\end{enumerate}
\end{enumerate}
Then $M$ is not left coherent.
\end{thm}

\begin{proof}
Let $l_a$ be the left congruence on $M$ generated by $(1,a)$, where $1$ is the identity of $M$, and consider the left congruence  $\lambda_{a,b}=\mathbf{l}(bl_a)$ on $M$. By Lemma \ref{lem:powers} we know that 
\[u\, \lambda_{a,b} \, v \Leftrightarrow uba^i=vba^j\]
for some $i,j\geq 0$. It follows that
$1\, \lambda_{a,b} \, e$ for all $e \in M \cap \bigcup_{i \geq 0} E_i$. For each $x \in F$ let $d(x)$ be defined as follows: if $x^{*} \in \bigcup_{i\geq 0} E_i$, set $d(x)={\rm min}\{k: x^{*} \in E_k\}$, and $d(x)=-1$ otherwise.
Suppose for contradiction that $X\subseteq M\times M$ is a finite symmetric generating set for $\lambda_{a,b}$. Since $X$ is finite, there exists $i  \in \mathbb{N}$ such that $i>d(x)$ holds for all $(x,y)\in X$. For the remainder of the proof we fix such a natural number $i \geq 1$.

For any $e \in M \cap E_i \subseteq E$, we have observed that $e\, \lambda_{a,b}\, 1$ and there is therefore an $X$-sequence from $e$ to $1$ and by condition (3)($ii$), $M \cap (E_{i} \setminus E_{i-1}) \neq \emptyset$. Let $e \in M \cap (E_{i} \setminus E_{i-1})$ be such that there is an $X$-sequence of length $\ell$ from $e$ to $1$,
\[e=t_1x_1, t_1y_1=t_2x_2, \ldots, t_\ell y_\ell=1\]
whilst for no element $e'$ of $M \cap (E_{i} \setminus E_{i-1})$ is there an $X$-sequence of length strictly less than $\ell$ from $e'$ to $1$.

Since $e$ is an idempotent in $M$ and $e=t_1x_1$ where $t_1, x_1 \in M$, it follows from (1) that
\begin{equation}\label{eqn:cu}
e=t_1^+x_1^{*}
\end{equation}
giving $e=t_1^{+}e$ and (using the fact that idempotents of $E$ commute) $e =x_1^{*}e$. Since $e \in E_i$ it is easy to see that $t_1^{+}, x_1^{*} \in E_i$ also. Since $(x_1, y_1) \in X$, by our choice of $i$ we have $i>d(x_1)$, giving $x_1^{*} \in E_{i-1}$. But then, since $E_{i-1}$ is a semigroup and $e \not \in E_{i-1}$, it follows from \eqref{eqn:cu} that  $t_1^{+} \in E_{i} \setminus E_{i-1}$. Now, since $(x_1,y_1)\in X$ we also have
\begin{equation}\label{eqn:cxy}
x_1ba^m=y_1ba^n    
\end{equation}
for some $m,n\geq 0$ and (by right multiplying by $a$ as necessary) we may assume that both $m$ and $n$ are greater than $i$. By our assumptions on $i,m$ and $n$ we note that $x_1^{*} \in E_{i-1} \subseteq E_m$. Since $F$ is right Ehresmann we have $x \,\widetilde{\mathcal{L}}_E \,x^*$,  and  as $\widetilde{\mathcal{L}}_E$ is a right congruence we deduce
\[y_1ba^n=x_1ba^m\,\widetilde{\mathcal{L}}_E\, x_1^* ba^m=ba^m.\]
It follows that $ba^m\,\leq_{\widetilde{\mathcal{L}}_E}\, ba^n$
and hence $m=n$, by condition (3)($i$). Thus by condition (2) we have that $t_1y_1\in E$, whilst by (3)($iii$), $y_1^*ba^n=ba^n$.

But then, as $t_1, y_1 \in M$, condition (1) gives $t_1y_1=t_1^{+}y_1^{*}$, where (by our previous observations) $t_1^{+} \in E_i \setminus E_{i-1}$ and $y_1^{*} \in E_{i-1}$, and so we find that $t_1y_1 \in E_i \setminus E_{i-1}$. But, since $t_1, y_1 \in M$,  this contradicts the choice of $e$ having  $X$-sequence to $1$ of shortest length amongst all elements of $M \cap (E_{i} \setminus E_{i-1})$. Thus $\lambda_{a,b}$ is not finitely generated, and $M$ is not left coherent.
\end{proof}

We state below for future convenience the left-right dual to Theorem \ref{thm:rightEhres}.

\begin{theorem}\label{thm:leftEhres} Let $F$ be a left $E$-Ehresmann monoid and let $M$ be a monoid subsemigroup of $F$ such that:
\begin{enumerate}[label=\textnormal{(\arabic*)}]
\item there exists a unary operation $x \mapsto x^* \in E$ on $F$ such that if $e=tx=e^2 \in E$, where $t,x\in M$ then $e=t^+x^*$;
\item if $tx \in E$ and  $wt=ws$ where $t,x,s,w \in M$, then $sx \in E$;
\item there exist elements $a,b\in M$ such that 
\begin{enumerate}[label=\textit{(\roman*)}]
\item $b,ab, a^2b,\ldots$ are all pairwise incomparable under the $\leq_{\widetilde{\mathcal{R}}_E}$-order of $F$;
\item for each $i\geq 0$, $M \cap (E_{i+1}
 \setminus E_{i}) \neq \emptyset$ where $E_i :=\{f \in E:  a^ibf=a^ib\}$;
\item for each $i \geq 0$, if $a^ibs=a^ibt$ and $a^ibt^+ = a^ib$, then $a^ibs^+=a^ib$.
\end{enumerate}
\end{enumerate}
Then $M$ is not right coherent.
\end{theorem}

\begin{remark}\label{rem:cond1} We have not said so explicitly, but if $F$ is a (two-sided) Ehresmann semigroup with respect to $E$, then a natural choice for $t^+$ in Theorem \ref{thm:rightEhres} is the idempotent in the $\widetilde{\mathcal{R}}_E$-class of $t$; likewise, a natural choice for $x^*$ in Theorem \ref{thm:leftEhres} is the idempotent in the $\widetilde{\mathcal{L}}_E$-class of $x$. In all applications of Theorem \ref{thm:rightEhres} (respectively of Theorem \ref{thm:leftEhres}) in this paper we will assume that we are taking this natural choice of $t^+$ (respectively, $x^*$). However,  we caution that, with respect to this choice, condition (1) need not hold in an arbitrary Ehresmann (or even inverse) monoid, as may be seen by considering a symmetric inverse monoid on a non-trivial set. 
\end{remark}

On the positive side, we now consider some natural circumstances in which the individual conditions of the Theorem \ref{thm:rightEhres} hold. 

\begin{lemma}\label{lem:(1)etc} Condition (1) of Theorem~\ref{thm:rightEhres} (and Theorem~\ref{thm:leftEhres})  holds for any monoid subsemigroup $M$ of  $F=Y\rtimes G$, where $Y$ is a monoid semilattice and $G$ is a group.
\end{lemma}
\begin{proof} It suffices to prove the statement in the case where $M=F$. If $(e,1_G)=(t,g_t)(x,g_x)$, then $e=t\wedge g_tx$ and $g_x^{-1}=g_t$, so
$(t,g_t)^+(x,g_x)^*=
(t,1_G)(g_x^{-1}x,1_G)=(t\wedge g_tx,1_G)=(e,1_G)$, as required.
\end{proof}
    
That conditions (2) and (3)($iii$) also hold for any submonoid $M$ of an inverse subsemigroup $F$ of the semidirect product $Y \rtimes G$ follows from Lemma~\ref{lem:cond2,3c}  below. To see this we first introduce some further terminology. 

\begin{definition} (Left $E$-adequate monoid)
We say that $S$ is \emph{left $E$-adequate} if $E$ is a semilattice of idempotents and each $\mathcal{R}^*$-class of $S$ contains a (necessarily, unique) idempotent of $E$; we denote the idempotent in the $\mathcal{R}^*$-class of $a$ by $a^+$. Equivalently, $S$ is left  $E$-adequate if it is left $E$-Ehresmann and $\mathcal{R}^*=\widetilde{\mathcal{R}}_E$. 
In the case where $E=E(S)$ we abbreviate  right $E$-adequate to \emph{left adequate}.
\end{definition}

Recall from Section \ref{sec:tools} that $F$ is said to be left abundant if every $\mathcal{R}^*$-class of $F$ contains an idempotent; left adequacy is therefore a strengthening of this condition.

\begin{remark}\label{rem:inverseeh} The ample identity has a strong influence on the structure of a semigroup, allowing idempotents to move to one side of a product. From the above if $S$ is left ample, then it is left adequate,  and also left $E$-adequate monoids are left $E$-Ehresmann.  However, left $E$-Ehresmann semigroups and in particular left adequate semigroups do not, in general, satisfy the ample identity. 
    \end{remark}

\begin{lemma} \label{lem:cond2,3c} Let $F$ be a right $E$-Ehresmann, left $E$-adequate monoid and $M$ a monoid subsemigroup of $F$. 
\begin{enumerate}
    \item If $E$ is a right unitary subset, then condition (2) of Theorem \ref{thm:rightEhres} holds.
    \item If $F$ is $E$-adequate, then condition (3)($iii$) of Theorem \ref{thm:rightEhres} holds.
\end{enumerate}\end{lemma} 
\begin{proof} \begin{enumerate} \item Let $e,t,x,y,w \in F$ be such that $e=tx=e^2 \in E$ and $xw=yw$. Since $F$ is left $E$-adequate  we have $xw^+=yw^+$ giving that $ew^+ = txw^+=tyw^+$ and it follows from unitariness that $ty \in E$ also. 
\item For any $x,y,w \in F$ with $xw=yw$ and $x^*w = w$, we have $xw^+=yw^+$. 
Using the fact that $x \,\mathcal{L}^*\,  x^*$, $y \,\mathcal{L}^*\,  y^*$ and $\mathcal{L}^*$ is a right congruence gives $x^*w^+ \,\mathcal{L}^*\,  xw^+ = yw^+\, \mathcal{L}^*\, y^*w^+$. Since each $\mathcal{L}^*$-class contains a unique idempotent, we must then have $x^*w^+ = y^*w^+$. Since $a^+a=a$ for all $a\in F$ we now deduce $w=x^*w=x^*w^+w=y^*w^+w=y^*w$. 
\end{enumerate}\end{proof}

From Theorems~\ref{thm:pi} and \ref{thm:rightEhres}, and Lemmas~\ref{lem:(1)etc} and \ref{lem:cond2,3c}, we can now deduce the following.

\begin{corollary}\label{cor:pi} Let $M$ be a monoid subsemigroup of an $E$-unitary inverse monoid $F$ (equivalently, of any semidirect product $Y\rtimes G$ where $Y$ is a semilattice and $G$ a group). Then $M$   satisfies conditions (1), (2) and (3)($iii$) of Theorem \ref{thm:rightEhres}.\\
Consequently, if $M$  contains elements $a,b \in M$ and $e_i \in E(M)$ for $i \geq 1$ satisfying:
\begin{enumerate}[label=\textit{(\roman*)}]
\item $b,ba, ba^2,\ldots$ are all pairwise incomparable under the $\leq_{\mathcal{L}}$-order of $F$,
\item for each $i\geq 1$, $e_iba^{i}=ba^{i}$ and $e_iba^{i-1}\neq ba^{i-1}$,
\end{enumerate}
then $M$ is not left coherent.
    \end{corollary}

As another consequence of Theorem \ref{thm:rightEhres} and Lemma ~\ref{lem:cond2,3c}  we have the following.

\begin{corollary}
\label{cor:E-adequate}
Let $F$ be an $E$-adequate monoid where $E \subseteq E(F)$ is a right unitary subset, and let $M$ be a  submonoid of $F$ such that:
\begin{enumerate}[label=\textnormal{(\arabic*)}]
\item if $e=tx=e^2 \in E$, where $t,x\in M$ then $e=t^+x^*$;
\item there exist elements $a,b\in M$ such that 
\begin{enumerate}[label=\textit{(\roman*)}]
\item $b,ba, ba^2,\ldots$ are all pairwise incomparable under the $\leq_{\widetilde{\mathcal{L}}_E}$-order of $F$,
\item for each $i\geq 0$, we have $M \cap (E_{i+1}
 \setminus E_{i}) \neq \emptyset$ where $E_i :=\{f \in E: fba^i=ba^i\}$,
\end{enumerate}
\end{enumerate}
Then $M$ is not left coherent.
\end{corollary}
    
We will give an application of Corollary \ref{cor:E-adequate} in the next section. We conclude this section by applying Corollary \ref{cor:pi} to $E$-unitary monoids of the form $\mathcal{S}(G)$ and to Margolis-Meakin expansions $M(G,X)$ (defined below) under certain assumptions on the group $G$, recovering several results from the literature along the way.

\begin{example}\label{ex:freemonoid}Let $G$ be a group containing an element $x$ of infinite order, and let $D(x)$ be the subsemigroup of $\mathcal{S}(G)$ generated by $a = (\{1,x^2\},x^2), b=(\{x\},1)$ and $e_i=(\{x^{2i}\}, 1)$ for $i \geq 1$. The elements $b, ba, ba^2, \ldots$ are  $\leq_{\mathcal{L}}$-incomparable in $\mathcal{S}(G)$ since $(ba^i)^* = (P_i, 1)$ where $P_0 = \{x\}$ and for $i \geq 1$, $P_i=\{x^{-2i+1}\} \cup \{x^{2(k-i)}: 0 \leq k \leq i\}$ are incomparable. The elements $e_i$ satisfy $e_iba^i = ba^i$ but $e_iba^{i-1} \neq ba^{i-1}$. Thus, Corollary \ref{cor:pi} applies to show that any monoid subsemigroup of $\mathcal{S}(G)$ containing $D(x)$ is not left coherent.
\end{example}

\begin{remark}
Taking $G$ as in the previous example and $N$ a submonoid of $G$ containing $x$, we find that the monoids $\mathcal{S}(N)$ and $\mathcal{S}^{\rm fin}(N)$ are not left coherent. This in particular recovers that fact proven in \cite{BGR23} that $\mathcal{S}^{\rm fin}(\mathbb{Z})$ is not left coherent. Of course the main result of \cite{BGR23} is much stronger: there any monoid \emph{containing $\mathcal{S}^{\rm fin}(\mathbb{Z})$ as a subsemigroup} is shown to be neither left nor right coherent. Thus if $M$ is any monoid containing a subgroup isomorphic to $\mathbb{Z}$ then it follows from \cite{BGR23} that $\mathcal{S}(M)$ is neither left nor right coherent. Example \ref{ex:freemonoid} allows us to show that $\mathcal{S}(N)$ is not left coherent in further situations, such as the case where $N$ is a free monoid.
\end{remark}

Since for any monoid $M$ we have that $M$ is a retract of $\mathcal{S}(M)$, it follows from \cite[Theorem 6.2]{GH} (see Theorem \ref{thm:retract} above) that if $M$ is not right (respectively left) coherent, then neither is $\mathcal{S}(M)$. On the other hand, we have seen that coherence of $M$ is not sufficient to guarantee coherence of $\mathcal{S}(M)$: if $M$ is a free group or a free monoid, then $M$ is coherent but $\mathcal{S}(M)$ is not \emph{left} coherent. It follows from the results of \cite{BGR23} that $\mathcal{S}(F_X)$ is not right coherent either. We show, by a slightly different method, using the left-right dual of Lemma \ref{lem:m=n}, that $\mathcal{S}(X^*)$ is not right coherent.

\begin{proposition}
\label{prop:triangle}
Let $M$ be a monoid containing an element $x$ such that the elements $1, x, x^2, \ldots$ form a strictly descending chain in the $\mathcal{R}$-order on $M$.
Then $\mathcal{S}(M)$ is not right coherent. In particular $\mathcal{S}(X^*)$ is not right coherent.
\end{proposition}
\begin{proof}
Let $T$ be the subset of $M$ containing all elements of the form $x^t$ where $t$ is an odd triangular number.  We show that the pair $a = (T,x^2)$ and $b = (\{1\}, x)$ satisfies the conditions of the left-right dual of Lemma \ref{lem:m=n}, and hence $\mathcal{S}(M)$ is not right coherent. For each $m\geq 0$ let us write $T_m = \bigcup_{0\leq i<m} x^{2i}T$, where $\emptyset = T_0 \subsetneq T_1=T \subsetneq T_2 \subsetneq \cdots$ are sets of odd powers of $x$,  with $a^mb = (T_m \cup \{x^{2m}\}, x^{2m+1})$. It is then clear that $(U,u), (V,v) \in \mathcal{S}(M)$ satisfy $a^nb(U,u) = a^mb(V,v)$ for some $m,n \geq 0$ if and only if $x^{2n+1}u=x^{2m+1}v$ and 
\begin{equation}\label{eq:tri}T_n \cup \{x^{2n}\} \cup x^{2n+1}U= T_m \cup \{x^{2m}\} \cup x^{2m+1}V.\end{equation}
If $(U,u)$ and $(V,v)$ satisfy this condition, then it follows from \eqref{eq:tri} that $m=n$ (if $n<m$, then by the assumption on $x$ and the fact that $T_m$ contains only odd powers of $x$ we see that $x^{2n}$ is contained in the left-hand set, but not the  right-hand set, and dually if $m<n$), and so the first condition of Lemma \ref{lem:m=n} holds.

For the second condition, as  the set of distances between the elements of $T$ is unbounded we have that the sets $T_0 \subsetneq T_1 \subsetneq \cdots$ of odd powers of $x$  are such that for each fixed $i\geq 1$ there exist (infinitely many) odd triangular numbers $t$ such that $x^{t-2}$ is not contained in $T_{i-1}$.
In particular, writing $t_i$ to denote the $i$th odd triangular number, it is straightforward to check that the $t_{2i+1}$ for $i \geq 1$ satisfy $x^{t_{2i+1}} \in T_i$, $t_{2i+1}\geq 2i+1 +t_{2i}$ and $x^{t_{2i+1}-2} \not\in T_{i-1}$. Then, for $u_i =(\{x^{t_{2i+1}-2i-1}\},1)$ and $v_i= (\emptyset,1)$ we see (using \eqref{eq:tri}) that $a^ibu_i = a^ibv_i$.
However, since $t_{2i+1}-2$ is odd (so in particular not equal to $2i-2$), and  $x^{t_{2i+1}-2}$ is not contained in $T_{i-1}$, we have that
$$T_{i-1} \cup \{x^{2(i-1)}\} \cup x^{2(i-1)+1}\{x^{t_{2i+1}-2i-1}\} = T_{i-1} \cup \{x^{2i-2}\} \cup \{x^{t_{2i+1}-2}\} \neq  T_{i-1} \cup \{x^{2i-2}\},$$
from which it follows (by using \eqref{eq:tri} once more) that $a^{i-1}bu_i \neq a^{i-1}bv_i$.
Thus for all $i \geq 1$ the elements $u_i= (\{x^{t_{2i+1}-2i-1}\},1)$,  $v_i = (\emptyset,1)$ satisfy condition (2) of Lemma \ref{lem:m=n}.
\end{proof}

There are many interesting monoid subsemigroups of semigroups $\mathcal{S}(G)$ where $G$ is a group to which we cannot apply Example \ref{ex:freemonoid}. Our next example is designed to be applied in some such situations, and recovers some known results. 

\begin{example}\label{ex:fi}
Let $G$ be a group with $g,h \in G$ where $g$ has infinite order and $h \not\in \langle g \rangle$, and let $D(g,h)$ be the subsemigroup of $\mathcal{S}(G)$ generated by $a = (\{1,g\},g), b=(\{1,h\},h)$ and $e_i=(\{1,h,hg,\ldots, hg^i\}, 1)$ for $i \geq 1$. The elements $b, ba, ba^2, \ldots$ are  $\leq_{\mathcal{L}}$-incomparable in $\mathcal{S}(G)$ since $(ba^i)^* = (P_i, 1)$ where $P_i = \{g^{-i}h^{-1}\} \cup \{g^k: -i \leq k \leq 0\}$ are incomparable. The elements $e_i$ satisfy $e_iba^i = ba^i$ but $e_iba^{i-1} \neq ba^{i-1}$. Thus, Corollary \ref{cor:pi} applies to show that any monoid subsemigroup of $\mathcal{S}(G)$ containing $D(g,h)$ is not left coherent. 
\end{example}

\begin{remark}\label{rem:fi} Recalling that $\FI(X)$, $\FA(X)$ and $\FLA(X)$ can each be viewed as a subsemigroup of $\mathcal{S}(F_X)$,  taking $g$ and $h$ to be distinct free generators of $F_X$ in Example \ref{ex:fi} recovers the negative results of \cite[Theorem 7.5]{GH}, that is, each of these monoids is not left coherent for $|X| \geq 2$.
\end{remark}

\begin{remark}\label{ex:BR} For any group $G$,  the expansion $Sz(G):=\{(A, g): 1,g \in A , A \subseteq G, |A|<\infty\}$ is an inverse monoid subsemigroup of $\mathcal{S}(G)$, with identity element $(\{1\}, 1)$. This monoid is the Birget-Rhodes expansion of $G$ considered by Szendrei in \cite{Sz89}. In the case where $G$ contains elements $g,h$ where $g$ has infinite order and $h \not\in \langle g \rangle$, then it follows from Example \ref{ex:fi} that $Sz(G)$ is not left coherent. Since $Sz(G)$ is inverse, neither is it right coherent.
\end{remark}

Another interesting class of $E$-unitary inverse monoids arises from the Margolis-Meakin expansion of groups; we give a brief definition and refer to  \cite{MM:1989} for any unexplained terminology.

\begin{definition}
\label{def:MM}(Margolis-Meakin expansion) Let $G$ be a group with generating set $X$ and ${\rm Cay}(G, X)$ the (right) Cayley graph of $G$ with respect to $X$, that is, the directed edge-labelled graph with set of vertices $G$ and a labelled directed edge $(h,x,hx)$ from $h$ to $hx$ with label $x$ for all $h \in G$ and all $x \in X$. There is a natural left action of $G$ on the set of edges of ${\rm Cay}(G,X)$ and this lifts to give an action by automorphisms on the monoid semilattice $(Y, \cup)$ of all finite subgraphs of ${\rm Cay}(G, X)$ where the identity element of $(Y, \cup)$
is the empty subgraph. By Theorem~\ref{thm:pi} the semidirect product $Y \rtimes G$ is an inverse monoid.  The Margolis-Meakin expansion of $G$ is then the 
subsemigroup of $Y\rtimes G$ given by:
$$M(G, X):=\{(P, g): P \in C, g \in V_P\},$$
where $C\subseteq Y$ consists precisely of those finite subgraphs that are \emph{connected and contain $1$ as a vertex}. It is not hard to see that $M(G,X)$ is a monoid with identity element $(1_C,1)$, where $1_C$ denotes the graph with single vertex $1$ and no edges, and moreover that $M(G, X)$ is an inverse subsemigroup of $Y \rtimes G$. (In the case where $G=F_X$ is the free group on this generating set, $M(G,X)$ is the free inverse monoid on $X$.)
\end{definition}

\begin{example}\label{ex:MM}
Let $X$ be a generating set for $G$ and suppose that $x,y \in X$ generate a free submonoid of $G$. For each $z \in \{x,y\}^*$,  denote by $P_z$ the directed graph whose vertices and edges are precisely those encountered on the unique path $\{x,y\}$-labelled path from $1$ to $z$. Then, by taking $a = (P_x,x)$, $b=(P_y,y)$ and $e_i = (P_{yx^i}, 1)$ it is easy to see that Corollary~\ref{cor:pi} can once again be applied in a similar manner to Example \ref{ex:fi} to show that $M(G,X)$ is not left coherent (and hence by duality not right coherent either). \end{example}
\section{Free Ehresmann monoids}
\label{sec:adequate}
In view of Remark \ref{rem:fi} several natural questions now arise. First, can our results (Theorem \ref{thm:rightEhres}, its corollaries, and their duals), be applied to determine non-(left/right) coherence of the free objects in the variety of (right/left/two-sided) Ehresmann monoids? Second, if any of these monoids are not left (respectively, right) coherent, do they satisfy any of the weaker conditions (e.g. weak left coherence, finitely left equated, left ideal Howson) discussed in Section \ref{sec:tools}? Coherence for free (left/right/two-sided) Ehresmann monoids will be the focus of the present section, and weak coherence will be considered in Section \ref{sec:weakly}.

We shall apply Corollary \ref{cor:E-adequate} and its left-right dual to demonstrate that the free Ehresmann monoid of rank at least $2$ is neither left nor right coherent. Moreover, we show that the free left Ehresmann monoid of rank at least $2$ is not left coherent. We note that the free objects in question, although considered here purely as monoids,  are actually free algebras in  varieties with augmented signature. Further, free (left) Ehresmann monoids coincide precisely with free (left) adequate monoids. We arrive at our results by using the structure of the  free (left) Ehresmann monoid given in \cite{K} and\cite{K2}. Specifically,  we may view the free (left)  Ehresmann monoid on a set $X$ as the set of (isomorphism types of) pruned (left-Ehresmann) $X$-trees with respect to the pruned operations: we recap the main notions here.

\begin{definition} (Pruned (left Ehresmann) trees and trunks)
An $X$-tree is a tree $T$ with directed edges labelled by elements of $X$, and with two distinguished vertices (the start
vertex and the end vertex) such that there is a (possibly empty) directed path from
the start vertex to the end vertex; this directed path is called the `trunk' of the tree (we will sometimes also refer to the element of $X^*$ labelling this path as the trunk). The graph with a single vertex marked as the start and end but with no edges is called the trivial $X$-tree, and is denoted by $1$.  A morphism between $X$-trees is a map taking edges to edges and vertices to vertices such that the orientation and label of each edge is preserved, and also the start and end vertex of the domain are mapped to the start vertex and end vertex of the co-domain. A retraction of a tree $T$ is an idempotent morphism from $T$ to itself; the image is a
retract of $T$. A tree $T$ is said to be pruned if it does not admit a (non-identity) retraction. We say that an $X$-tree $T$ is said to be left Ehresmann
if for each vertex $v$ of $T$ there is a directed path from the start vertex to $v$. 
\end{definition}

\begin{remark}
Notice that a morphism maps the trunk edges of its domain bijectively onto the trunk edges of its
image, and hence also that every retraction fixes all trunk edges.
\end{remark}

In what follows, we will write simply `tree' in place of $X$-tree. Moreover, trees will be viewed up to isomorphism: the set of all isomorphism types of $X$-trees is denoted by ${\rm UT}^1(X)$, the set of all isomorphism types of left Ehresmann $X$-trees is denoted by ${\rm LUT}^1(X)$, the set of all isomorphism types of pruned $X$-trees is denoted by $\FAd(X)$, and the set of all isomorphism types of pruned left Ehresmann $X$-trees is denoted by $\FLAd(X)$.

 \begin{remark}(Branches and pruning)
A `branch' of $T$ is a subgraph that contains no trunk edges and contains all the vertices and edges on the non-trunkward side of a fixed edge.  For each $T \in {\rm UT}^1(X)$ there is a unique up to isomorphism pruned tree $\overline{T}$ -- called the pruning of $T$-- that is a retraction of $T$. The ‘pruning’ process can be viewed as a process of pruning branches: if there is a branch of $T$ which can be ‘retracted into’ (i.e. mapped by a retraction which fixes
everything except the given branch) the rest of the tree, we can remove it, and moreover, repeating this process until no such branches exist will give $\overline{T}$.
\end{remark}
\begin{definition} (Unpruned and pruned operations)
Let $S, T \in {\rm UT}^1(X)$. The `unpruned product' of $S$ and $T$ denoted by $S \times T$ is (the isomorphism type of) the tree obtained by glueing together
$S$ and $T$ by identifying the end vertex of $S$ with the start vertex of $T$ and keeping all
other vertices and all edges distinct. For each $ T \in {\rm UT}^1(X)$,  $T^{(+)}$ is (the isomorphism
type of) the tree with the same labelled graph and start vertex as $T$, but with end vertex
equal to the start vertex of $T$, whilst $T^{(*)}$ is (the isomorphism
type of) the tree with the same labelled graph and end vertex as $T$, but with start vertex
equal to the end vertex of $T$. For pruned trees $S, T \in \FAd(X)$ there are corresponding pruned operations defined by $ST = \overline{S \times T}$, $T^+ =\overline{T^{(+)}}$, and $T^* =\overline{T^{(*)}}$.
\end{definition}

We summarise some further key points from \cite{K} and \cite{K2} concerning these definitions:
\begin{itemize}
    \item unpruned multiplication is an associative binary operation on ${\rm UT}^1(X)$ with identity element given by the tree with single vertex, making ${\rm UT}^1(X)$ a bi-unary monoid with respect to the two unary operations $T \mapsto T^{(+)}$ and $T \mapsto T^{(*)}$;
     \item pruned multiplication is an associative binary operation on $\FAd(X)$ with identity element given by the tree with single vertex, making $\FAd(X)$ a bi-unary monoid, with respect to the unary operations $T \mapsto T^{+}$ and $T \mapsto T^{*}$, and the map $T \mapsto \overline{T}$ is a surjective morphism of bi-unary monoids;
    \item $\FAd(X)$ is the free Ehresmann monoid on $X$, where each $a \in X$ is identified with the tree with a single edge from start vertex to end vertex labelled by $a$;
    \item the relevant multiplication and `plus' operations restrict to ${\rm LUT}^1(X)$ and $\FLAd(X)$, making them unary monoids, and moreover $\FLAd(X)$ is the free left Ehresmann monoid on $X$;
    \item the idempotents of $\FAd(X)$ and $\FLAd(X)$ are precisely those pruned trees whose start and end vertices are equal.
\end{itemize}
We now state the main result of this section.

\begin{theorem}\label{cor:adequate}
Let $|X| > 1$. Then the free Ehresmann monoid $\FAd(X)$ is neither left nor right coherent. The free left Ehresmann monoid $\FLAd(X)$ is not left coherent.
\end{theorem}
\begin{proof}
Since $\FAd(X)$ is also the free object in the quasi-variety of adequate monoids it is in particular adequate. Moreover, it is easy to see that the set of idempotents of $\FAd(X)$ is unitary. Thus, our aim is to apply Corollary~\ref{cor:E-adequate} in the case where $F$ is the free Ehresmann monoid $\FAd(X)$ and $M$ is either the free Ehresmann monoid $\FAd(X)$ or the free left Ehresmann monoid $\FLAd(X)$; that the latter may be viewed as a submonoid of $\FAd(X)$ follows from \cite[Theorem 3.18]{K2}. We show that conditions (1) and (2) of Corollary \ref{cor:E-adequate} hold.

If $t,x,e$ are pruned $X$-trees such that $tx=e=e^2$ in the free Ehresmann monoid it then follows that there is a morphism $\phi$ from the unpruned tree $t \times x$ to $e$. Notice that by construction every trunk edge of $t$ is also a trunk edge of $t \times x$, and likewise, every trunk edge of $x$ is a trunk edge of $t \times x$. Since any morphism of $X$-trees maps trunk edges of its domain bijectively onto the trunk edges of its image, we now see that $t$ and $x$ have no trunk edges, in other words $t$ and $x$ are idempotent. But then $t=t^+$ and $x=x^*$, giving that condition (1) of Corollary \ref{cor:E-adequate} holds when taking $F$ to be the free Ehresmann monoid and $M$ any submonoid of $F$.

Now, suppose $X = \{a,b,\ldots\}$ and
identify the generator $a$ with the pruned tree consisting of a single directed edge labelled by $a$, with start and end vertex marked to agree with the orientation of the edge, and likewise for $b$. It is clear that for any $i \geq 0$ the pruned tree obtained from $ba^i$ is a directed path labelled by the word $ba^i$, and by construction $(ba^i)^*$ is a right identity for $ba^i$. However, it is straightforward to check that for any $i \neq j \geq 0$, $(ba^j)^*$ is not a right identity for $ba^i$, demonstrating that the elements $b, ba, ba^2, \ldots$ are incomparable in the $\leq_{\widetilde{\mathcal{L}}}$-order of $F$. Thus (2)($i$) holds for $M=F=\FAd(X)$. Similarly, for any $i \geq 0$, $(ba^{i+1})^+$ is a left identity for $ba^{i+1}$, but not a left identity for $ba^{i}$, demonstrating that condition (2)($ii$) also holds for $M=F=\FAd(X)$, and hence $\FAd(X)$ is not left coherent. By similar reasoning, one finds that the left-right dual of Corollary \ref{cor:E-adequate} may be applied to demonstrate that $\FAd(X)$ is not right coherent either.

Finally, we note that since each of the elements $ba^i$ and $(ba^{i+1})^+$ for $i \geq 0$ also lie in the free left Ehresmann monoid $\FLAd(X)$, exactly the same argument applies to show that conditions (2)($i$) and (2)($ii$) hold for $M=\FLAd(X)$ and $F=\FAd(X)$, and hence the free left Ehresmann monoid $\FLAd(X)$ is not left coherent either.
\end{proof}

The above result is a neat application of Corollary \ref{cor:E-adequate}; a very natural question at this point is whether the result can be obtained by other means. We show next that this result cannot be obtained by a straightforward application of existing results from the literature concerning subsemigroups and retracts. First, note that \cite[Theorem 4]{BGR23} (detailed above as Theorem \ref{thm:forbid}) does not apply in this situation, since the elements $g,h$ specified in the set-up of that  theorem generate a group, whilst the free Ehresmann monoid (and the free left Ehresmann monoid) does not contain any non-trivial groups. Next, we note that the free left ample monoid (and hence also the free ample monoid) of any rank is not isomorphic to a submonoid of any free  Ehresmann monoid (as we show below)
 and so in particular is not a retract of the free   Ehresmann  monoid.
Thus Theorem \ref{cor:adequate} cannot be obtained from the existing negative results for left coherency of the free (left) ample monoid \cite[Theorem 7.5]{GH} together with the fact that retraction respects coherency \cite[Theorem 6.3]{GH}.

\begin{proposition}
Let $X$ and $Y$ be arbitrary non-empty sets. The free left ample monoid $\FLA(X)$  does not embed into the free Ehresmann monoid $\FAd(Y)$. 
\end{proposition}
We emphasise that we are regarding  $\FLA(X)$ and $\FAd(Y)$ as monoids, not as unary monoids; the argument for unary monoids would be somewhat more straightforward than that given below.

\begin{proof}
Suppose that $\FLA(X)$ embeds into $\FAd(Y)$ via $\theta$. Fix $x \in X$; it is clear that the  image of $x$ under this embedding is not idempotent and hence has trunk containing an edge labelled by some generator of  $ Y$. For each $T \in \FAd(Y)$, let $d(T)$ be the maximal length of any undirected path beginning at the start vertex of $T$ and let $c(T) \subseteq Y$ be the set of edge labels of $T$.

Let $W = x\theta$ and for each $i \geq 0$ let $F_i = (x^i)^+\theta$.  Since $x(x^i)^+ = (x^{i+1})^+x$ and $(x^{i+1})^+(x^i)^+ = (x^{i+1})^+$ hold in $\FLA(X)$ for all $i \geq 0$ we must have $WF_i = F_{i+1}W$ and $F_{i+1}F_i = F_{i+1}$ for all $i \geq 0$. In particular, each idempotent $F_i$ is a subtree of $F_{i+1}$ and hence satisfies $d(F_i) \leq d(F_{i+1})$ and $c(F_i)\subseteq c(F_{i+1})$.

Suppose first that $\bigcup_{i\geq 0}c(F_i)$ is infinite. Then there must be an $N\geq 1$ such that $c(F_N)\cup  c(W)
\subsetneq c(F_{N+1})\cup  c(W)$. This clearly contradicts $WF_N = F_{N+1}W$.

Now assume that $\bigcup_{i\geq 0}c(F_i)=Z$, where $Z$ is finite. For each   $\ell \geq 0$ the number of pruned trees $T \in \FAd(Z)$ with the property that $d(T)=\ell$ is finite. It follows that there exists an infinite subsequence of the idempotents $F_{i}$ such that  $0<d(F_{i_1}) <d(F_{i_2}) < \cdots$. Thus we may choose $N > 0$ such that $d(W) < d(F_N)< d(F_{N+1})$.
Consider the product $F_{N+1}W$. Since  $F_{N+1}$ is idempotent and $d(W) < d(F_{N+1})$ there exists an undirected path $p$ of length $d(F_{N+1})$ from the start vertex of $F_{N+1}$ that cannot be pruned in the product $F_{N+1}W$. Thus we may view $p$ as a path in $F_{N+1}W$. Notice that, since $F_{N+1}$ is idempotent, $p$ does not intersect the trunk of $F_{N+1}W$. Now consider the product $WF_{N}$ and let $q$ be a path from the start vertex that does not intersect the trunk. Since $q$ does not intersect the trunk, it follows that $q$ visits only the vertices of $W$ giving $|q| \leq d(W) < d(F_{N+1}) = |p|$ hence contradicting that $WF_N = F_{N+1}W$.
\end{proof}

\section{Weak coherence in free left Ehresmann monoids}
\label{sec:weakly}
The aim of this subsection is to demonstrate that the free left Ehresmann monoid of any rank is weakly coherent. To prove this we make use of an alternative description of the free left Ehresmann  monoid given in \cite{BGG} and \cite{GG}.

\begin{definition}(Normal forms)
Let $T \in \FLAd(X)$. Then $T$ can be uniquely written in the form
$$T = {t_0}e_1{t_1} \cdots e_m {t_m},$$
where $m \geq 0$, $e_1, \ldots, e_m$ are non-trivial idempotents (i.e. non-identity idempotents; the corresponding pruned tree contains at least one edge), ${t_1}, \ldots, {t_{m-1}} \in X^+$,  ${t_0}, {t_m} \in X^*$ and for $1 \leq i \leq n$, $e_i < ({t_i}e_{i+1}\cdots e_n{t_m})^+$. We call this unique expression the normal form of $T$. 
\end{definition}

\begin{remark}
\label{rem:algo}
In the proof of \cite[Lemma 3.1]{BGG} an algorithm is provided to convert an element of $\FLAd(X)$ written as a word over the alphabet $B_X:=X^* \cup \mathcal{E}_X$, where $\mathcal{E}_X$ denotes the set of non-trivial idempotents of $\FLAd(X)$, into normal form. The procedure to convert $d_1 \cdots d_\ell\in B_X^*$ to its normal form is as follows.
\begin{enumerate}
\item[(0)] Remove any $d_j=1_{X^*}$ with $0<j<\ell$ from $d_1\dots d_{\ell}$
and relabel the word as $c_1\cdots c_k$.
    \item[(I)] If $c_{t-1}=u'$ and $c_t= u''$ for some $t \geq 1$ and  $u',u''$ in $X^*$, or in $\mathcal{E}_X$, replace  $c_{t-1}c_t$ by $u \in B_X$ where $u=u'u''$. After all such replacements we arrive at a word $$v_0 f_1 v_1 f_2 \cdots  f_n v_n \in B_X^*,$$ where $n \geq 0$ and $v_0, v_n \in X^*$, $v_1, \ldots, v_{n-1} \in X^+$, and $f_t\in \mathcal{E}_X$ for $1 \leq t \leq n$.
    \item[(II)] For each idempotent $f$ in the current expression, let $b_{f} \in B_X^*$ denote the suffix sitting to the immediate right of $f$. Then working from right to left, if $fb_{f}^+ = b_{f}^+$ then we may remove $f$, otherwise we may replace $f$ by  $fb_{f}^+$; notice that after this adjustment we arrive at a word of the same overall form, but containing fewer idempotents $f'$ with $f'b_{f'} \neq f'$. Thus working from right to left we will produce a normal form.
\end{enumerate} 
\end{remark}

The recipe to obtain a normal form from a pruned tree $T$ may be described visually as follows. Take the factors lying in $X^*$ to be the trunk of $T$, factored according to the position of any branches of $T$ that cannot be pruned, and iteratively define the idempotents $e_i$ (working from right to left) to be the maximal  pruned tree that can be placed in this position and still yield the
same pruned tree  after composition with the existing right factor.

\begin{lemma}
Let $A, T \in \FLAd(X)$ with normal forms given by $A = {a_0}e_1{a_1}\cdots {a_{n-1}}e_n{a_n}$ and $T = {t_0}f_1{t_1}\cdots {t_{m-1}}f_m{t_m}$. \begin{enumerate}[label=\textnormal{(A\arabic*)}]
    \item If $a_nt_0 \neq 1$, then the normal form of $AT$ is  ${p_0}g_1{p_1} \cdots {p_k}g_k{w}f_1{t_1}\cdots f_m{t_m}$ for some ${p_1}, \ldots {p_k}, w\in X^+$, $p_0 \in X^*$, and non-trivial idempotents $g_1, \ldots, g_k$ where $t_0$ is a right factor of $w$.
    
    \item If $a_nt_0 = 1$, the normal form of $AT$ is ${p_0}g_1{p_1} \cdots g_{k-1}{p_k}d{t_1}\cdots f_m{t_m}$ for some ${p_1}, \ldots {p_k} \in X^+$, $p_0 \in X^*$, and non-trivial idempotents $d, g_1, \ldots, g_{k-1}$, where $f_1 \geq d$.
\end{enumerate}
\end{lemma}

\begin{proof}
(1) Suppose first that $a_nt_0 \neq 1$. Let $k$ be minimal such that $e_k$ is not deleted in the  rewriting algorithm applied to $AT$. Then completing the process we will obtain
$$AT = {p_0}g_1{p_1} \cdots {p_k}g_k{w}f_1{t_1}\cdots f_m{t_m},$$
where $w={a_k\cdots a_n}t_0$ and $g_k = e_k ({a_k}\cdots {a_nt_0}f_1 \cdots f_m{t_m})^+= e_k (wf_1 \cdots f_m{t_m})^+$. 

(2) If $a_nt_0 = 1$, then $a_n=t_0=1$ and applying the rewriting algorithm above gives:
$$AT = {p_0}g_1{p_1} \cdots {p_k}d{t_1}\cdots f_m{t_m},$$
where $d= e_nf_1$, noting that $e_nf_1 \leq f_1 <({t_1}f_2\cdots f_m{t_m})^+$. 
\end{proof}

\begin{proposition}
\label{lem:FLEisweaklyleft}
The free left Ehresmann  monoid  $\FLAd(X)$ is left ideal Howson.
\end{proposition}
\begin{proof} We show that the intersection of any two principal left ideals of the free left Ehresmann monoid $\FLAd(X)$ is either empty or principal.
Let $M = \FLAd(X)$ and let $T, S \in M$ have normal forms $T = {t_0}f_1{t_1}\cdots {t_{m-1}}f_m{t_m}$ and $S = {s_0}e_1{s_1}\cdots {s_{n-1}}e_n{s_n}$, and suppose that $MS \cap MT$ is not empty. We consider three cases:

Case (i): Suppose that there exist $A, B \in M$ such that $AT=BS$ and the normal forms of $AT$ and $BS$ are
\begin{eqnarray*}
AT &=& {p_0}g_1{p_1} \cdots {p_k}g_k{w}f_1{t_1}\cdots f_m{t_m}\\
BS &=&{q_0}h_1{q_1} \cdots {q_l}h_l v e_1{s_1}\cdots e_n{s_n},
\end{eqnarray*} where $w=w't_0$ and $v=v's_0$. Without loss of generality, assume that $m \leq n$ and write $n=m+r$ for some $r \geq 0$. If $r>0$ then by uniqueness of normal forms we have ${s_r} = w$, $e_{r+1} = f_1$, ${s_{r+1}} = {t_1}, \ldots, e_n=f_m$ and ${s_n} = {t_m}$. But then 
$$S = {s_0}e_1{s_1}\cdots {s_{n-1}}e_n{s_n} = {s_0}e_1{s_1}\cdots e_rw' T,$$ and so $MS \cap MT = MT$. On the other hand, if $r=0$ then $m=n$ and $w= v$, $f_1 = e_1$, ${t_1} = {s_1}, \ldots, f_m=e_n, {t_m}={s_n}$. Since $w't_0 = v's_0$ there exists ${u} \in X^*$ such that either ${t_0} = {u} {s_0}$, in which case $T = {u}S$, and so $MS \cap MT  = MS$ -  or else ${s_0} = {u} {t_0}$, in which case $S = uT$, and so $MS \cap MT  = MT$.

Case (ii) Suppose that there exist $A, B \in M$ such that $AT=BS$ and the normal forms of $AT$ and $BS$ are
\begin{eqnarray*}
AT &=& {p_0}g_1{p_1} \cdots {p_k}g_k{w}f_1{t_1}\cdots f_m{t_m}\\
BS &=&{q_0}h_1{q_1} \cdots h_{l-1}{q_l}d{s_1}\cdots e_n{s_n},
\end{eqnarray*}
where $w=w't_0$ and $d \leq e_1$. Recall in this case that we must have $s_0=1$.  If $m=n$ then by uniqueness of normal forms we have ${q_{l}} = w$, $d = f_1$, ${s_{1}} = {t_1}, \ldots, e_n=f_n$ and ${s_n} = {t_n}$. But then 
$$T = {t_0}f_1{t_1}\cdots {t_{n-1}}f_n{t_n} = {t_0}d{s_1}\cdots e_n{s_n} = {t_0}de_1{s_1}\cdots e_n{s_n}  = t_0dS,$$ and so $MS \cap MT = MS$. If $n=m+r$ for some $r>0$ then ${s_r} = w$, $e_{r+1} = f_1$, ${s_{r+1}} = {t_1}, \ldots, e_n=f_m, {s_n}={t_m}$, giving $S = e_1 {s_1}\cdots {s_{r-1}}e_rw'T$ and hence $MS \cap MT = MT$. Finally, if $m=n+r$ for some $r>0$, we have $f_{r+1} = d$, ${t_{r+1}} = {s_1}, \ldots, f_m=e_n, {t_m}={s_n}$, giving $T = t_0f_1t_1 \cdots t_rdS$, and so $MS \cap MT = MS$.

Case (iii): Suppose that whenever $U=AT=BS$, we have the following normal forms
\begin{eqnarray*}
AT &=& {p_0}g_1{p_1} \cdots g_{k-1}{p_k}d_A{t_1}\cdots f_m{t_m}\\
BS &=&{q_0}h_1{q_1} \cdots h_{l-1}{q_l}d_B{s_1}\cdots e_n{s_n},
\end{eqnarray*}
where $d_A \leq f_1$ and $d_B \leq e_1$.
Recall that in this case we must have $t_0=s_0=1$. Without loss of generality, we may assume that $m \leq n$ and write $n=m+r$ for some $r \geq 0$. If $r>0$ then by uniqueness of normal forms we have $e_{r+1} =d_A$, ${s_{r+1}} = {t_1}, \ldots, e_n=f_m$ and ${s_n} = {t_m}$. But then 
$$S = e_1{s_1}\cdots {s_r}d_A{t_1} \cdots f_m {t_m} = e_1{s_1}\cdots {s_r}d_Af_1 {t_1} \cdots f_m {t_m} =e_1{s_1}\cdots {s_r}d_A T,$$ and so $MS \cap MS  = MT$. 

Suppose now that  $r=0$. Then $m=n$, $l=k$, and we have ${p_i} = {q_i}$,  ${s_i}={t_i}$ and $g_i=h_i$ for all $i$, $d_A=d_B$ and $e_i = f_i$ for all $i>1$. Set $e = e_1f_1$. We show that $MS \cap MT = MeT$. To see this, first note that since $d_A=d_B$ we have $d_A \leq e_1$ and hence $d_A\leq e$.  By the assumption of this case, for each $U =AT=BS \in MS \cap MT$ we may therefore write
$$U = {p_0}g_1{p_1} \cdots g_{l-1}{p_l}d_A{t_1}\cdots f_m{t_m} = {p_0}g_1{p_1} \cdots g_{l-1}{p_l}d_Aee_1{t_1}\cdots f_m{t_m}  = {p_0}g_1{p_1} \cdots g_{l-1}{p_l}d_AeT$$
demonstrating that $MS \cap MT \subseteq MeT$. Finally, recalling that $s_0=t_0=1$ we have
\begin{eqnarray*}
eT &=& e_1f_1(f_1{t_1}\cdots {t_{m-1}}f_m{t_m}) = e_1f_1({t_1}\cdots {t_{m-1}}f_m{t_m})\\ &=& e_1f_1({s_1}\cdots {s_{m-1}}e_m{s_m}) = e_1f_1(e_1{s_1}\cdots {s_{m-1}}e_m{s_m}) = eS,
\end{eqnarray*}
from which it follows that $MeT = MeS \subseteq MS \cap MT$. Thus $MS \cap MT = MeT$.
\end{proof}

That the free left Ehresmann monoid is finitely right equated can also be shown by working with normal forms.

\begin{proposition}
\label{lem:FLEisFRE}
The free left Ehresmann monoid $\FLAd(X)$ is finitely right equated.
\end{proposition}

\begin{proof}
Let $T = {t_0}f_1 \cdots f_m {t_m}$ be in normal form and suppose that $(U, V) \in \mathbf{r}(T)$. It follows from the discussion above that $U$ and $V$ must have `a large amount of agreement' in their normal forms; we utilise this to show that $\mathbf{r}(T)$ is finitely generated.

If $m=0$ then we have $T={t_0}$ and  a trivial calculation reveals that $U=V$ from which it is immediate that $\mathbf{r}(T)$ is finitely generated.

Suppose then from now on that $m \geq 1$. In what follows, let us write $U = {u_0}e_1 \cdots e_n {u_n}$ and $V = {v_0}e'_1 \cdots e'_l {v_l}$ in normal form and assume that $n \geq l$ writing $n=l+r$ for some $ r\geq 0$. Proceeding as before,the normal forms of $TU$ and $TV$ may be written as:
\begin{eqnarray*}
TU &=& {w_k}g_{k}\cdots {w_1}g_1{w_0}g_0{u_1}e_2\cdots e_n{u_n}\\
TV &=& {z_j}h_{j}\cdots {z_1}h_1{z_0}h_0{v_1}e'_2 \cdots e'_l{v_l},
\end{eqnarray*}
where $k,j \geq 0$, each $g_i, h_i$ is a non-trivial idempotent, all ${w_i}, {z_i} \in X^*$ with only $w_k, z_j$ possibly equal to $1$, and further:
\begin{itemize}
    \item[(I)] if $t_mu_0 \neq 1$, then $g_0=e_1$ and ${w_0} = t_a \cdots t_m u_0$ for some $0 \leq a \leq m$;
    \item[(II)] if $t_mu_0 = 1$, then $g_0 =f_me_1$ and ${w_0} = t_a \cdots t_{m}$ for some $0 \leq a \leq m$;
    \item[(I$'$)] if  $t_mv_0 \neq 1$, then $h_0=e'_1$ and ${z_0} = t_b \cdots t_m v_0$ for some $0 \leq b \leq m$;
    \item[(II$'$)] if $t_mv_0 = 1$, then $h_0 =f_me'_1$ and ${z_0} = t_b \cdots t_{m}$ for some $0 \leq b \leq m$.
\end{itemize}

By uniqueness of normal forms,  it follows immediately that (in all cases) we have:
\begin{enumerate}[label=\textnormal{(\arabic*)}]
    \item $t_0\cdots t_m u_0 u_1 \cdots u_n =w_k\cdots w_1w_0u_1 \cdots u_n = z_j\cdots z_1z_0v_1 \cdots v_l =t_0\cdots t_m v_0 v_1 \cdots v_l$ (by comparing the `trunks' of the two expressions) and
    \item $u_n = v_l, \ldots, u_{r+2} = v_2, u_{r+1} = v_1$,  and $e_n = e'_l,\ldots , e_{r+2}=e'_2$
    (by reading both expressions from right to left and comparing the first few factors),
    \item $
v_0 = u_0 u_1 \cdots u_r$ (by cancellation, using (2) and the equality between the first and last expressions in (1)). 
\end{enumerate}

\medskip
If $t_m \neq 1$, we show that $U=V$. Indeed, in this case we have that $t_mu_0 \neq 1 \neq t_mv_0$, and it follows from (I) and (I$'$) above that we have $g_0=e_1$, $h_0=e'_1$, and $w_0 = t_a\cdots t_mu_0$ and $z_0 = t_b\cdots t_mv_0$ for some $0 \leq a,b \leq m$. Thus continuing the comparison of factors begun in (2), we immediately obtain $u_n = v_l, \ldots, u_{r+2} = v_2, u_{r+1} = v_1, e_n = e'_l, \ldots, e_{r+2} = e'_2$ and $e_{r+1} = e'_1$. Notice that if $r=0$, then by (3) $u_0=v_0$ and hence $U=V$. Suppose then for contradiction that $r>0$. Then continuing our comparison of factors, we must also have $z_0 = u_r$, giving $t_b\cdots t_mv_0 = u_r$. But then substituting $v_0 = u_0 u_1 \cdots u_r$  gives $t_b\cdots t_mu_0 u_1 \cdots u_r = u_r$, and by cancellation we obtain $t_b\cdots t_mu_0 u_1 \cdots u_{r-1} = 1$, contradicting that $t_m \neq 1$. Thus if $m \geq 1$ and $t_m \neq 1$ we have shown that $TU=TV$ implies $U=V$, from which it  follows that $\mathbf{r}(T)$ is finitely generated.

\medskip
From now on we suppose that $m \geq 1$ and $t_m=1$, aiming to show that in this case $\mathbf{r}(T)$ is finitely generated. Suppose that $U$ and $V$ are as above with $TU=TV$. We now divide into cases based on whether $u_0=1$ or not and whether $v_0 = 1$ or not. Notice that if $v_0=1  \neq u_0$, we obtain an immediate contradiction to (3), so this situation cannot occur. This leaves three possibilities to explore.

\medskip
If $u_0 \neq 1 \neq v_0$, then arguing exactly as in the case where $t_m \neq 1$ using (I) and (I$'$), we obtain $u_n = v_l, \ldots, u_{r+2} = v_2, u_{r+1} = v_1, e_n = e'_l, \ldots, e_{r+2} = e'_2$ and $e_{r+1} = e'_1$. If $r>0$ we again obtain $t_b\cdots t_mu_0 u_1 \cdots u_{r-1} = 1$, contradicting that $u_0 \neq 1$. Thus $r=0$, $u_0=v_0$ by (3), and hence $U=V$.

\medskip
If $u_0=1  \neq v_0$, then in this case by (II) and (I$'$) we have $g_0 =f_me_1$, $h_0=e'_1$, $w_0 = t_a \cdots t_{m}$, and $z_0 = t_b \cdots t_{m}v_0$, for some $0 \leq a,b \leq m$. Notice also that it follows immediately from (3) that $r>0$, since $u_0 \neq v_0$. Continuing our comparison of factors from (2) we have $e_{r+1} = h_0 = e'_1$ and $z_0=u_r = t_b \cdots t_m v_0$  which by (3) gives $u_r = t_b \cdots t_mu_0\cdots u_r$. Since $r>0$ and $u_i \neq 1$ for all $0<i<r \leq n$, it follows from the last equation that  $r=1$, giving $t_b\cdots t_mu_0=1=t_b\cdots t_m$ and $z_0=v_0$. Further, note that (by comparisons of factors in the normal forms) we must also have $g_0=h_1$,  $u_1=z_0=v_0$ and $e_2=h_0=e'_1$.
By (2), (3) and the previous lines we then have $u_i = v_{i-1}$ for all $1 \leq i \leq n=l+1$, $e_i = e'_{i-1}$ for all $2 \leq i \leq n=l+1$. In particular, recalling that $u_0=1$ in this case, $U =e_1V$. Since $z_0 = v_0$, it follows from the construction of our normal form that $f_mV^+ \neq V^+$ (otherwise $z_0$ would contain $v_0$ as a proper factor; however we have seen that these are equal) and so $h_1 = f_mV^+$. Hence 
$f_me_1V^+ = g_0V^+=h_1V^+ = f_mV^+$.  Since $V^+ \mathcal{R}^* V$ the latter gives $f_mU = f_me_1V = f_mV$.

\medskip
Finally, if $t_m =u_0 = v_0 = 1$, then by (II) and (II$'$) we immediately have $g_0 =f_me_1$, $h_0 =f_me'_1$,  $w_0 = t_a \cdots t_{m}$, and $z_0 = t_b \cdots t_{m}$, for some $0 \leq a,b \leq m$. Recalling that $r\geq 0$ and $u_i \neq 1$ for all $0<i< n$, we also deduce from (3) that either $r=0$ or else $r=n=1$ and $u_0=u_1=1$. Suppose first that $r=0$. Then by the assumption of this case and (2) we have $u_i=v_i$ for $0 \leq i \leq n=l$ and $e_i=e'_i$ for $2 \leq i \leq n=l$. Thus ${u_1}e_2 \cdots e_n{u_n} = {v_1}e'_2 \cdots e'_n{v_n}$. Calling this element $\beta$  and continuing our comparison of factors of the normal form, we have that $f_me_1=g_0=h_0 =f_me'_1$. Thus
$U = e_1 \beta$ and $V = e_1'\beta$, where  $f_me_1\beta^+ =f_me'_1\beta^+$.  Since $\beta^+ \mathcal{R}^* \beta$ the latter gives $f_mU =f_me_1\beta = f_me'_1\beta = f_mV$. On the other hand, if $r=n=1$ and $u_0=u_n=0$, then we find that $l= 0$, $V=1$, $U=e_1$  and hence $t_0 f_1 t_1 \cdots t_m f_m = T= TV=TU= Te_1 = t_0 f_1 t_1 \cdots t_m f_me_1$. In particular, we must have $f_mV =f_m=f_me_1 =f_mU$.

\medskip
Now, notice that for $t_m=1$ we have in all cases  deduced that either $U=V$ or else that $f_mU=f_mV$, from which it follows that $\mathbf{r}(T) \subseteq \mathbf{r}(f_m)$. Conversely, if $(P,Q) \in \mathbf{r}(f_m)$, then we have $f_mP=f_mQ$, and since $t_m=1$ this gives $TP = {t_0}f_1 \cdots f_m {t_m}P = {t_0}f_1 \cdots f_mP = {t_0}f_1 \cdots f_mQ = {t_0}f_1 \cdots f_m {t_m}Q = TQ$, showing that $\mathbf{r}(T) = \mathbf{r}(f_m) = \langle (1,f_m)\rangle$.
\end{proof}

\begin{proposition}
\label{lem:FLEisweaklyright}
The free left Ehresmann monoid $\FLAd(X)$ is right ideal Howson.
\end{proposition}
\begin{proof}
 Let $M:= \FLAd(X)$ where elements of $\FLAd(X)$ are viewed once more as pruned trees, and let $\mathcal{E}_X^1$ denote the set of all idempotents of $M$ (including $1$). For each pruned tree $U \in M$ let us write let $d(U)$ to denote the maximal length of a directed path from the start vertex in the pruned tree $U$. Then we may write down a unique factorisation of $U$ over the alphabet $X \cup \mathcal{E}_X^1$ as follows:
\begin{equation}
\label{fact}
e_{U,0}x_{U,1} e_{U,1} \cdots e_{U, l-1} x_{U,l} e_{U,l}
\end{equation}
where $d(U) \geq l \geq 0$, for each $i \geq 1$, $x_{U,i} \in X$ are such that $x_{U,1}\cdots x_{U,l}$ is the trunk of $U$, and $e_{U,i}$ is the idempotent corresponding to the unprunable branches (if any) at position $i$ along the trunk.  Colour the edges of $U$ as follows: if there exists a tree $U'$ such that edge $e$ of $U$ can be pruned in computation of the product $UU'$, then colour $e$ blue. Otherwise, colour $e$ red. Notice that the trunk of $U$ is red and moreover each branch of $U$ is either completely blue or completely red. From the colouring of $U$ we may further factorise each $e_{U,k}$ as a product of incomparable idempotents
$$e_{U,k} = r_{U,k}\Pi_{i \in I_k} b_{U,k,i}$$ where $r_{U,k}$ is the idempotent corresponding to the (unprunable in every $UU'$) red tree at position $k$ along the trunk and the $b_{U,k,j}$ are the individual (unprunable in $U$, but prunable in some $UU'$) blue branches connected at position $i$ along the trunk and indexed by a finite set $I_k$. Notice that all branches of the final idempotent $e_{U,l}$ will be coloured blue, since these edges can all be pruned in computation of the product $Ue_{U,l}$, and so $r_{U,l} =1$.

Now suppose that $S$ and $T$ are two pruned trees in $M$ and 
that $TM \cap SM$ is non-empty. It follows immediately that the trunks of $S$ and $T$ must be comparable in the $\mathcal{R}$-order on $X^*$. Without loss of generality, let us suppose that the trunk of $T$ is $t=x_1\cdots x_n \cdots x_m$ and the trunk of $S$ is $s = x_1 \cdots x_n$ where $m \geq n \geq 0$ and $x_i \in X$ for all $1 \leq 
i \leq m$. By the previous paragraph we may factorise $T$ and $S$ as: 
\begin{eqnarray*}T &=& (r_{T,0} \Pi_{i \in I_0} b_{T,0,i})x_1(r_{T,1}\Pi_{i \in I_1} b_{T,1,i})x_2 \cdots x_n (r_{T,n}\Pi_{i \in I_n} b_{T,n,i}) \cdots x_m (\Pi_{i \in I_{m}} b_{T,m,i}),
\\
S &=& (r_{S,0} \Pi_{i \in J_0} b_{S,0,i})x_1(r_{S,1}\Pi_{i \in J_1} b_{S,1,i})x_2 \cdots x_n (\Pi_{i \in J_n} b_{S,n,i}),\end{eqnarray*}
where each bracketed expression gives a factorisation of the unique idempotent at the corresponding position along the trunk according to the colouring of its branches.

Let $U \in TM \cap SM$, with unique factorisation over $X \cup \mathcal{E}_X^1$ as given in equation \eqref{fact}. Since $U = TT'=SS'$ for some $T', S' \in M$, from the pruning process we may write
\begin{equation}
\label{eq:U}U=e_0x_1e_1\cdots e_{q} x_{q+1} \cdots e_{m-1}x_mA, \mbox{ where } q = {\rm min}(n,m-1) \mbox{ and }
\end{equation}
\begin{itemize}
    \item[(a)] $A \in M$ with $e_{T,m}A = A$ and if $m=n$  also $e_{S,n}A = A$, 
    \item[(b)] for $0 \leq k \leq m-1$, $e_k = r_{T,k} \Pi_{i \in P_k} b_{T, k, i}$ where $P_k = \{i \in I_k: b_{T, k, i} \ngeq (x_{k+1}e_{k+1}\cdots x_mA)^+\}$,
\item[(c)] for $0 \leq k \leq q$, 
$e_k = g_k \Pi_{i \in Q_k} b_{S, k, i},$
where $Q_k = \{i \in J_k: b_{S, k, i} \ngeq (x_{k+1}e_{k+1}\cdots x_mA)^+\}$ and where $g_k = r_{S,k}$ if $k<n$ whilst if $k=n\leq m-1$ then either $g_n=1$ or $g_n$ is incomparable to $\Pi_{i \in Q_n} b_{S, n, i}$ and satisfies $g_n \ngeq (x_{n+1} e_{n+1}\cdots x_mA)^+$.
\end{itemize}

Conversely, suppose that $U$ is as in equation \eqref{eq:U} with factors satisfying condition (a)-(c) above. Then from the factorisation of $T$  and the fact that $e_{T,m}A = A$ we see that the (unpruned) product of $T$ and $A$ may be written as:
\begin{eqnarray*}
TA &=& e_{T,0} \, x_1e_{T,1} \, \cdots x_ne_{T,n} \cdots x_me_{T,m}A\\
&=&e_0e'_0\, x_1e_1e'_1 \cdots x_q e_qe'_q \cdots e_{m-1}e_{m-1}'x_mA, 
\end{eqnarray*}
where for $k=0,\ldots, m-1$, $e_k' = \Pi_{i \in I_k\setminus P_k}b_{T, k,i}$, so that $e_{T,k} = e_ke_k'$. From the definition of $P_k$, each $e_k'$ satisfies $e_k'(x_{k+1}e_{k+1}\cdots x_mA)^+ = (x_{k+1}e_{k+1}\cdots x_mA)^+$. Thus after pruning these idempotents we obtain:
\begin{eqnarray*}
TA &=&e_0x_1e_1 \cdots x_q e_q \cdots x_m A= U \in TM.
\end{eqnarray*}
By a similar argument, writing $e_k'' = \Pi_{i \in J_k \setminus Q_k}b_{S, k,i}$ so that $e_{S,k} = e_ke_k''$ for all $0 \leq k \leq q$, we find that if $n<m$, then $g_ne_n = e_n$ and hence
\begin{eqnarray*}
Sg_nx_{n+1}e_{n+1}\cdots x_mA &=&e_0e_0''x_1e_1e_1''x_2 \cdots x_n e_ne_n''g_n x_{n+1} \cdots x_mA\\
&=&e_0x_1e_1x_2 \cdots x_n e_n x_{n+1} \cdots x_m e_mA = U \in SM,
\end{eqnarray*}
whilst if $n=m$ we have $e_{S,n}A = A$ and hence
\begin{eqnarray*}
SA &=&e_0e_0''x_1e_1e_1''x_2 \cdots x_{n-1}e_{n-1}e_{n-1}''x_ne_{S,n}A\\
&=&e_0x_1e_1x_2 \cdots x_m e_mA = U \in SM.
\end{eqnarray*}
Let us write $C =g_nx_{n+1}\cdots x_m$ if $n<m$ and $C = 1$ if $n=m$, so that in all cases $SCA = U$.

Now let $Y \subseteq X$ be the finite set of elements that label an edge in either $T$ or $S$, let $M_Y$ denote the subset of $M$ consisting of all pruned trees whose labels are drawn from the set $Y$, and,  as before, for each $V \in M_Y$ let $d(V)$ denote the maximal length of a directed path from the start vertex in the pruned tree $V$. Then $M_Y \cong \FLAd(Y)$ and for each $d \geq 0$ the set $B_{d} = \{V \in M_Y: d(V)\leq d\}$ is finite. In particular, the set $Z_L = TM \cap SM \cap B_{L}$, where $L = {\rm max}(d(T), d(S))$ is finite.  We claim that for each $U \in TM \cap SM$, 
there exist $U'\in Z_{L}$  and $U'' \in M$, such that $U = U'U''$, giving that $Z_{L}$ is a finite generating set for $TM \cap SM$. 
To see this, let $U$ be as in \eqref{eq:U}, so that in particular $U=TA=SCA$, with $A$ and $C$ as defined above, and let $f$ be the idempotent subtree of $A^+$ 
consisting of those edges that either some edge of $T$ can retract onto in computation of the product $TA^+$ or some edge of $S$ can retract onto in the computation of the product $SCA^+$. Note that the blue edges of $T$ that are pruned in the product $Tf$ are precisely the same as the blue edges of $T$ that are pruned in the product $TA^+$, which in turn are precisely the blue edges of $T$ that are pruned in the product $TA$. Likewise, the blue edges of $S$ that are pruned in the product $SCf$ are precisely the same as the blue edges of $S$ that are pruned in the product $SCA^+$, which in turn are precisely the blue edges of $S$ pruned in the product $SCA$. Moreover, by construction, 
$$fA = fA^+A = A^+A = A,$$
and hence $U = TA = TfA$ and $U = SCA = SCfA$.
Then setting $U' = e_0x_1e_1x_2 \cdots x_mf$ we see that $U' \in M_Y$,  $d(U')\leq L$ and $U' = Tf = SCf \in TM \cap SM$, giving $U = U'U''$ for $U' \in Z_L$ and $U'' = A \in M$.
\end{proof}

\begin{theorem}
\label{thm:FLEisweaklycoherent}
The free left Ehresmann  monoid  $\FLAd(X)$ is weakly coherent, that is, it is weakly left and weakly right coherent. 
\end{theorem}
\begin{proof}
Since $\FLAd(X)$ is the free object on $X$ in the quasi-variety of left adequate monoids, it is in particular left adequate, and hence also left abundant. By Proposition~\ref{prop:rabundweak}, $\FLAd(X)$ is finitely left equated. Proposition ~\ref{lem:FLEisweaklyleft} now gives that $\FLAd(X)$ is  weakly left coherent. The result for weak right coherence follows similarly from Propositions~\ref{lem:FLEisFRE} and ~\ref{lem:FLEisweaklyright}.\end{proof}

\section{An embedding of the free Ehresmann monoid}
\label{sec:embed}
In this section we show that the free Ehresmann monoid $\FAd(X)$ embeds into an inverse monoid of the form $Y\rtimes G$, where $Y$ is a semilattice and $G$ is a group acting by automorphisms.  Subsequently,  Corollary \ref{cor:pi} and its left-right dual can also be applied to show that $\FAd(X)$ is neither left nor right coherent. Since $\FLAd(X)$ is a submonoid of $\FAd(X)$, this also gives an alternative proof that $\FLAd(X)$ is not left coherent. We begin with an important remark. Since $\FAd(X)$ is not ample, it cannot embed into a monoid of the form $Y\rtimes G$ (nor, indeed, into any inverse or ample monoid) in a way that preserves the two unary operations ($^+$ and $^\ast$)  since this would force the ample identities to hold in $\FAd(X)$. In other words, $\FAd(X)$ does not embed into  any $Y\rtimes G$ as a bi-unary monoid. Similarly, $\FLAd(X)$ does not embed into any $Y\rtimes G$ as a unary monoid since $\FLAd(X)$ does not satisfy the left ample identity. It is perhaps surprising that $\FAd(X)$ embeds into some $Y\rtimes G$ as a monoid.  Our approach makes use of the fact that $\FAd(X)$ is isomorphic to a particular quotient of the monoid free product of the free monoid generated by $X$ with the semilattice of idempotents $\FAd(X)$.

\begin{definition}(Left and right actions on the semilattice of idempotents)
Let $E(X)$ denote the semilattice of idempotents of $\FAd(X)$. Note that $E(X)$ acts (by multiplication) on the left and right of $E(X)$. Viewing $E(X)$ and $X^*$ as submonoids of $\FAd(X)$ in the obvious way, we may also define left and right actions of $X^*$ on $E(X)$ by utilising the operations of $\FAd(X)$: for a word $w \in X^*$ and an idempotent $f \in E(X)$ we define $w \cdot f = (wf)^+$ and $f \circ w = (fw)^*$.
\end{definition}

\begin{definition}(Construction of Branco, Gomes and Gould)
Now viewing $X^* \cup E(X)$ as a union of two  disjoint monoids (which we call parts for short) the semigroup free product of $X^*$ and $E(X)$ is the set
$$X^*\ast E(X):= \{ a_1 \cdots a_m: m \geq 1 \mbox{ and } \forall i,  a_i \in X^* \cup E(X)  \mbox{ and }  a_i, a_{i+1} \mbox{not in the same part} \},$$
where the product of elements $u=a_1 \cdots a_m$ and $v=b_1 \cdots b_k$ is given by $a_1 \cdots a_mb_1 \cdots b_k$ if $a_m$ and $b_1$ are in different parts, or by $a_1 \cdots a_{m-1}cb_2 \cdots b_k$, where $c=a_mb_1$ if $a_m$ and $b_k$ are in the same part. The (left and right) actions of $X^*$ and $E(X)$ on $E(X)$ extend to give left and right) actions of $X^* \ast E(X)$ on $E(X)$, also denoted by $\cdot$ and $\circ$. Let $1_{X^*}$ and $1_{E(X)}$ denote the identity elements of $X^*$ and $E(X)$ respectively, and for each $u \in X^* \ast E(X)$ write $u^+ = u \cdot 1_{E(X)}$ and $u^* = 1_{E(X)} \circ u$.  By \cite[Theorem 6.1]{BGG} the free Ehresmann monoid is then isomorphic to the quotient
$$P(X^*, E(X)):= (X^* \ast E(X))/ 
H^\sharp,$$
where  
$H:=\{(u^+u,u), (uu^*,u): u \in X^* \ast E(X)\} \cup \{(1_X^*, 1_{E(X)})\}$.
\end{definition}

\begin{remark}(The free Ehresmann monoid via a monoid free product)
\label{rem:quotient}
The monoid free product of $X^*$ and $E(X)$ is by definition the quotient $C_X:=(X^* \ast E(X))/\{(1_X^*, 1_{E(X)})\}^\sharp$.  For each $u \in X^* \ast E(X)$, we write $[u]_C$ to denote the corresponding element of $C_X$. Let $\mathcal{E}_X = E(X) \setminus \{1_{E(X)}\}$.  Writing
\begin{eqnarray*}
C_{t,t}&:=&\{ a_1 a_2 \cdots a_{2m+1}: m \geq 0, a_{i} \in X^+ \mbox{ if $i$ is odd, } a_{i} \in \mathcal{E}_X \mbox{ if $i$  is even}  \}\\
C_{t,f}&:=&\{ a_1 a_2 \cdots a_{2m}: m \geq 1, a_{i} \in X^+ \mbox{ if $i$  is odd, } a_{i} \in \mathcal{E}_X \mbox{ if $i$  is even}\}\\
C_{f,t}&:=&\{ a_1 a_2 \cdots a_{2m}: m \geq 1, a_{i} \in \mathcal{E}_X \mbox{ if $i$  is odd, } a_{i} \in X^+ \mbox{ if $i$  is even}\}\\
C_{f,f}&:=&\{ a_1 a_2 \cdots a_{2m+1}: m \geq 0,a_{i} \in \mathcal{E}_X \mbox{ if $i$  is odd, } a_{i} \in X^+ \mbox{ if $i$  is even}\},
\end{eqnarray*}
we see that $\{1_{E(X)}\} \cup C_{t,t} \cup C_{t,f} \cup C_{f,t} \cup C_{f,f}$ is a transversal of the congruence classes. For each element $\gamma \in C_X$ there is a unique word $t \in X^*$ formed by omitting all interior factors in $E(X) \cup \{1_{X^*}\}$ from any alternating product $u \in X^* \ast E(X)$ with $\gamma = [u]_C$; we shall call this the trunk word of $\gamma$ and we will denote the empty trunk by $1$.  Note that we may also define unary operations on $C_X$ as follows:
$$[u]_C^+ = [u \cdot 1_{E(X)}]_C \mbox{ and }[u]_C^* = [1_{E(X)} \circ u]_C,$$
and by construction both  $[u]_C^+$ and $[u]_C^*$ have trunk word $1$. The free Ehresmann monoid is then isomorphic to the quotient $C_X/ 
J^\sharp,$
where  
$$J:=\{(\gamma^+\gamma,\gamma), (\gamma\gamma^*,\gamma): \gamma \in C_X\}.$$
For $\gamma \in C_X$ we write $[\gamma]_F$ for the corresponding congruence class. 
\end{remark}

We shall need some technicalities to introduce the semilattice $Y_X$ we use to exhibit our embedding. To this end, we next define a morphism $\theta$ from the monoid free product $C_X$ to an auxiliary semidirect product $Z_X \rtimes F_X$, where $F_X$ is the free group on $X$ and $Z_X$ is a semilattice; ultimately, our intended semilattice $Y_X$ will be defined as a quotient of $Z_X$ determined by $\theta$. We begin with the definition of $Z_X$ and a technical lemma required to define the morphism $\theta$. 

\begin{definition}(The semilattice $Z_X$ and elements  $\tau$)
Let $S_X$ denote the free (monoid) semilattice generated by $y_{x,h}$ and $e_{f,h}$ where $x \in X$, $f \in \mathcal{E}_X$, $h \in F_X$ and let $Z_X$ denote the quotient $S_X/Q^\sharp$ where $Q=\{(e_{f',h}e_{f'',h}, e_{f'f'',h}): f',f'' \in \mathcal{E}_X, h \in F_X\}$. For a word $w = x_1 x_2\cdots x_n \in X^+$ and $h \in F_X$, write:
$$\tau_{w,h} := y_{x_1, h}y_{x_2, hx_{1}} \cdots y_{x_n, hx_{1}x_2 \cdots x_{n-1}} \in Z_X.$$
\end{definition}

\begin{lemma}
\label{lem:tauprod}
Let $w, v \in X^+$ and $h \in F_X$.
Then $\tau_{v,h}$ is a factor of $\tau_{w,1}$ in $Z_X$ if and only if $h \in X^*$ and there exists $u \in X^*$ such that $w = hvu$.  Moreover, if $h=1$ then  $\tau_{w, 1} = \tau_{v, 1}\tau_{u, v}$, whilst if $h \in X^+$ then $\tau_{w, 1} = \tau_{h,1} \tau_{v, h}\tau_{u, hv}$.
\end{lemma}
\begin{proof}
Let $w=x_{1}\cdots x_j$ and $v = x'_1\cdots x'_k $  where $x_i, x'_\ell\in X$ for $1\leq i\leq j$ and $1\leq \ell\leq k$. Let  $p_i$   denote  the length $i$ 
prefix of $w$ for $0\leq i\leq j-1$, and let  $p_\ell'$ denote the length $\ell$ prefix of $v$ for $0\leq \ell\leq k-1$. Then 
$$\tau_{w, 1} = y_{x_1, p_0} y_{x_2, p_1} \cdots y_{x_j, p_{j-1}}
\mbox{ and }\tau_{v, h} = y_{x_1', hp_0'} y_{x_2', hp_1'} \cdots y_{x_k', hp_{k-1}'}.$$
A factor $z$ of $\tau_{w,1}$ is then precisely a product $z= \Pi_{i \in I} y_{x_i, p_{i-1}}$ for some $I \subseteq\{1, \ldots, j\}$. Thus  we see that $z=\tau_{v,h}$ is a factor of $\tau_{w,1}$ if and only if for some such $I$ we have 
 \[\Pi_{\ell=1}^{k} y_{x'_\ell, hp'_{\ell-1}} = \Pi_{i \in I} y_{x_i, p_{i-1}}.\]
 By comparing subscripts of the $y$'s it follows that  this can only happen if $I = \{i, i+1, \ldots, i+k-1 \}$, $v = x_i\cdots x_{i+k-1}$ and $h=p_{i-1} \in X^*$.  That is, $\tau_{v,h}$ is a factor of $\tau_{w,1}$ if and only if $h$ is a proper prefix of $w \in X^+$ and there exists $u \in X^*$ such that $w = hvu$. If $h=1$ it is straightforward to verify that $\tau_{w,1}= \tau_{v,1}\tau_{u,v}$, and similarly if $h\in X^+$ one finds that $\tau_{w, 1} = \tau_{h,1} \tau_{v, h}\tau_{u, hv}$.
\end{proof}

\begin{definition}(The semidirect product $Z_X \rtimes F_X$)
The free group $F_X$ acts on the left of $S_X$ by morphisms as follows: $g \cdot y_{x,h} = y_{x,gh}$ and $g \cdot e_{f,h} = e_{f,gh}$. In particular, $g \cdot \tau_{w,h} = \tau_{w,gh}$ for all $w \in X^+$ and all $g,h \in F_X$. For each relation $(e_{f',h}e_{f'',h},e_{f'f'',h}) \in Q$ we also have $(g\cdot e_{f',h} g\cdot e_{f'',h}, g\cdot e_{f'f'',h}) \in Q$, and so we then have an induced action of $F_X$ by monoid morphisms on $Z_X$.  Thus $Z_X \rtimes F_X$ is an inverse monoid with identity element $(1_{Z_X}, 1)$.    
\end{definition}

Observe that for any $w,w'\in X^+$ we have
\[(\tau_{w,1},w)(\tau_{w',1},w')=(\tau_{w,1}\, w\cdot \tau_{w'1},ww')=(\tau_{w,1}\tau_{w',w},ww')=(\tau_{ww',1}, ww'),\]
and for any $f,f'\in \mathcal{E}_X$
\[(e_{f,1},1)(e_{f',1},1)=(e_{f,1}e_{f',1},1)
=(e_{ff',1},1).\]
It follows (with adjustment for the identity) that we have monoid morphisms from $X^*=X^+\cup 1_{X^*}$ and 
$E(X)=\mathcal{E}_X\cup 1_{E(X)}$ to $Z_X \rtimes F_X$. Using the properties of the monoid free product, we may now define a morphism from $C_X$ to $Z_X \rtimes F_X$.

\begin{definition}(The morphism $\theta$)
 The morphism $\theta: C_X \rightarrow Z_X \rtimes F_X$ is  determined by extending the mapping:
\begin{eqnarray*}
[1_{E(X)}]_C& \mapsto &  (1_{Z_X}, 1),\\ \
[w]_C & \mapsto & (\tau_{w,1}, w) \mbox{ for all } w \in X^+,\\ \
[f]_C& \mapsto & (e_{f,1}, 1) \mbox{ for all } f \in \mathcal{E}_X.\end{eqnarray*}
\end{definition}

\begin{lemma}
Let $t_i \in X^+$ and $f_i \in \mathcal{E}_X$ for $1 \leq i \leq k$. Then 
\begin{enumerate}
    \item $[t_1f_1\cdots t_kf_k]_C\theta = (\tau_{t,1}\Pi_{i=1}^ke_{f_i,t_1\cdots t_{i}}, t),$
where $t=t_1\cdots t_k$.
\item $[t_1f_1\cdots t_kf_kt_{k+1}]_C\theta = (\tau_{t',1}\Pi_{i=1}^ke_{f_i,t_1\cdots t_{i}}, t'),$
where $t'=t_1\cdots t_{k+1}$.
\item $[f_0t_1f_1\cdots t_kf_kt_{k+1}]_C\theta = (\tau_{t',1}\Pi_{i=0}^ke_{f_i,t_1\cdots t_{i}}, t'),$
where $t'=t_1\cdots t_{k+1}$.
\item $[f_0t_1f_1\cdots t_kf_k]_C\theta = (\tau_{t,1}\Pi_{i=0}^ke_{f_i,t_1\cdots t_{i}}, t),$
where $t=t_1\cdots t_{k}$.
\end{enumerate}
\end{lemma}
\begin{proof}
(1) For $k= 1$ we have
$$[u]_C\theta = (\tau_{t_1,1}, t_1)(e_{f_1,1}, 1) = (\tau_{t_1,1}e_{f_1,t_1}, t_1),$$
as required. Suppose for some $k \geq 1$ we have $[t_1f_1\cdots t_kf_k]_C\theta =
(\tau_{t,1}\Pi_{i=1}^ke_{f_i,t_1\cdots t_{i}}, t)$, where $t=t_1 \cdots t_k$. Then since $\theta$ is a morphism
\begin{eqnarray*}
[t_1f_1\cdots t_kf_kt_{k+1}f_{k+1}]_C\theta &=&  [t_1f_1\cdots t_kf_k]_C\theta [t_{k+1}f_{k+1}]_C\theta\\ &=&
(\tau_{t,1}\Pi_{i=1}^ke_{f_i,t_1\cdots t_{i}}, t)(\tau_{t_{k+1},1}e_{f_{k+1},t_{k+1}}, t_{k+1})\\ &=&(\tau_{t,1}\Pi_{i=1}^ke_{f_i,t_1\cdots t_{i}} \tau_{t_{k+1}, t} e_{f_{k+1},tt_{k+1}}, t t_{k+1})\\
&=& (\tau_{t',1} \Pi_{i=1}^{k+1}e_{f_i,t_1\cdots t_{i}}, t'),
\end{eqnarray*}
where $t'=t_1\cdots t_{k+1}$, and the equality in the final line follows from the previous lemma.

(2) Using the fact that $\theta$ is a morphism together with part (1) and the previous lemma gives
\begin{eqnarray*}
[t_1f_1\cdots t_kf_kt_{k+1}]_C\theta &=&[t_1f_1\cdots t_kf_k]_C\theta [t_{k+1}]_C\theta = (\tau_{t,1}\Pi_{i=1}^ke_{f_i,t_1\cdots t_{i}}, t)(\tau_{t_{k+1}, 1},t_{k+1})\\
&=&(\tau_{tt_{k+1},1}\Pi_{i=1}^ke_{f_i,t_1\cdots t_{i}}, tt_{k+1}) = (\tau_{t',1}\Pi_{i=1}^ke_{f_i,t_1\cdots t_{i}}, t'),
\end{eqnarray*}
where $t'=t_1\cdots t_{k+1}$.

(3) Using the fact that $\theta$ is a morphism together with part (2) gives
\begin{eqnarray*}
[f_0t_1f_1\cdots t_kf_kt_{k+1}]_C\theta &=&[f_0]_C\theta[t_1f_1\cdots t_kf_kt_{k+1}]_C\theta = (e_{f_0, 1},1)(\tau_{t',1}\Pi_{i=1}^ke_{f_i,t_1\cdots t_{i}}, t')\\
&=& (\tau_{t',1}\Pi_{i=0}^ke_{f_i,t_1\cdots t_{i}}, t')
\end{eqnarray*}
where again $t'=t_1\cdots t_{k+1}$.

(4) Similarly, using the fact that $\theta$ is a morphism together with part (1) gives
\begin{eqnarray*}
[f_0t_1f_1\cdots t_kf_k]_C\theta &=&[f_0]_C\theta[t_1f_1\cdots t_kf_k]_C\theta = (e_{f_0, 1},1)(\tau_{t,1}\Pi_{i=1}^ke_{f_i,t_1\cdots t_{i}}, t)\\
&=& (\tau_{t,1}\Pi_{i=0}^ke_{f_i,t_1\cdots t_{i}}, t)
\end{eqnarray*}
\end{proof}

\begin{remark}\label{rem:split}
For $[u]_C \in C_X$ write $[u]_C\theta = ([u]_C\theta_1, [u]_C\theta_2)$. By the previous lemma we have, in all cases, that $[u]_C\theta_2$ is the trunk word of $[u]_C$, and moreover $\theta_2$ is the morphism from $C_X$ to $F_X$ that maps each element to its trunk word. The function $\theta_1$ is {\em not} a morphism however: for example, if $t_1, t_2 \in X^+$ then $[t_1t_2]_C\theta_1 = \tau_{t_1t_2,1}$ and $[t_1]_C\theta_1[t_2]_C\theta_1 = \tau_{t_1,1} \tau_{t_2,1}$, and by Lemma \ref{lem:tauprod} $\tau_{t_2, 1}$ is not a factor of $\tau_{t_1t_2,1}$.
\end{remark}

\begin{definition}\label{defn:positions}
Let $\gamma \in C_X$ have trunk word $t$, let $p$ be a prefix of $t$ and let $f \in \mathcal{E}_X$. We say that $f$ left-splits $\gamma$ at $p$ if there exists a factorisation $\gamma = \gamma' [f]_C \gamma''$ in $C_X$ where $\gamma' = [u']_C$ for some $u' \in C_{t,t} \cup C_{f,t} \cup \{1_{X^*}\}$ with trunk word $p$, and $\gamma'' = [u'']_C$ for some $u'' \in C_{t,t} \cup C_{t,f} \cup \{1_{X^*}\}$. Note that if such a factorisation exists, the idempotent $f \in \mathcal{E}_X$ is unique. We write $P(\gamma)$ for the set of prefixes $p$ of $t$ for which there exists an idempotent  that left-splits $\gamma$ at $p$, and we write $l(p, \gamma)$ to denote this unique idempotent. 
\end{definition}

The next remark follows from Remark~\ref{rem:split} and Definition~\ref{defn:positions}.

\begin{remark}\label{rem:theta_1}  For $\gamma\in C_X$ we have
\[\gamma \theta_1 = \tau_{c,1} \Pi_{p \in P(\gamma)} e_{l(p,\gamma),p}.\]\end{remark}

\begin{lemma}\label{lem:thetaprod}
Let $\gamma, \delta \in C_X$. Then
$$(\gamma\delta)\theta_1 = \gamma\theta_1 \, (\gamma\theta_2 \cdot (\delta\theta_1)).$$
Consequently, $(\gamma^+\gamma)\theta_1=\gamma^+\theta_1\, \gamma\theta_1$.
\end{lemma}

\begin{proof}
Let $c$ be the trunk of $\gamma$ and $d$ the trunk of $\delta$. In view of Remark~\ref{rem:theta_1},  we have
\begin{eqnarray*}
\gamma \theta_1 &=& \tau_{c,1} \Pi_{p \in P(\gamma)} e_{l(p,\gamma),p}\\
\delta \theta_1 &=& \tau_{d,1} \Pi_{p \in P(\delta)} e_{l(p,\delta),p}\\
(\gamma\delta) \theta_1 &=& \tau_{cd,1} \Pi_{p \in P(\gamma\delta)} e_{l(p,\gamma\delta),p}.
\end{eqnarray*}
Note that $P(\gamma\delta) = P(\gamma) \cup cP(\delta)$ and  for all $p \in P(\delta)$ we have $l(p, \delta) = l(cp, \gamma\delta)$. Together with Lemma \ref{lem:tauprod} and commutativity in $Z_X$ this gives:
\begin{eqnarray*}
(\gamma\delta) \theta_1 &=& \tau_{cd,1} \Pi_{p \in P(\gamma\delta)} e_{l(p,\gamma\delta),p} \\
&=& \tau_{c,1}\tau_{d,c} \Pi_{p \in P(\gamma)} e_{l(p,\gamma),p}\Pi_{p \in P(\delta)} e_{l(p,\delta),cp}\\
&=& (\tau_{c,1} \Pi_{p \in P(\gamma)} e_{l(p,\gamma),p})\,( c\cdot (\tau_{d,1} \Pi_{p \in P(\delta)} e_{l(p,\delta),p}))\\
&=&\gamma\theta_1 \, (\gamma\theta_2 \cdot (\delta\theta_1)).
\end{eqnarray*}

The final statement follows from the fact that $\gamma^+\theta_2=1$.
\end{proof}

In what follows, we aim to determine a set of relations $R$ to be imposed on $Z_X$, giving a quotient $Y_X:=Z_X/R^\sharp$ with the property that the morphism $\theta: C_X \rightarrow Z_X \rtimes F_X$ induces a morphism $\hat\theta:C_X/J^\sharp \rightarrow Y_X \rtimes F_X.$ 

\begin{definition}(The semilattice $Y_X$)  Define $Y_X$ to be the quotient $Z_X/R^\sharp$ where $R=R_1\cup R_2$ and
\begin{eqnarray*} R_1&:=&\{(h \cdot (\gamma^+\theta_1\gamma\theta_1), \,\, h\cdot( \gamma\theta_1)): \gamma \in C_X, h \in F_X\}\\ 
R_2&:=&\{(h \cdot (\gamma\theta_1 (\gamma\theta_2 \cdot \gamma^*\theta_1)), h \cdot (\gamma\theta_1) ): \gamma \in C_X, h \in F_X\}.
\end{eqnarray*}
For $z \in Z_X$ we write $[z]_Y$ for the corresponding congruence class. 
\end{definition}

Note that from the form of the relations in $R$, the action of $F_X$ on $Z_X$ induces an action of $F_X$ on $Y_X$.

\begin{proposition}
For all $\gamma \in C_X$, let 
$$\gamma\hat\theta = ([\gamma\theta_1]_Y, \gamma\theta_2) \in Y_X \rtimes F_X.$$
Then $\hat\theta$ induces a monoid morphism $\hat\theta: \FAd(X) \rightarrow Y_X \rtimes F_X$.
\end{proposition}

\begin{proof}By Remark \ref{rem:quotient}, we may view $\FAd(X)$ as the quotient $C_X/J^\sharp$ and write elements in the form $[\gamma]_F$ where $\gamma \in C_X$. We claim that 
$$[\gamma]_F \hat\theta = ([\gamma\theta_1]_Y, \gamma\theta_2)$$
is a well-defined monoid morphism. 

Suppose that $\gamma_1, \gamma_2 \in C_X$ with $[\gamma_1]_F = [\gamma_2]_F$; we aim to show $([\gamma_1\theta_1]_Y, \gamma_1\theta_2) = ([\gamma_2\theta_1]_Y, \gamma_2\theta_2)$. By induction, it suffices to show that this is the case when $\gamma_2$ is obtained by applying a single relation from $J$ to $\gamma_1$. Suppose first that
$\gamma_1 = \delta\gamma\delta'$ and $\gamma_2 = \delta\gamma^+\gamma\delta'$. Since $\theta_2$ is a morphism and $\gamma^+\theta_2 =1$ we have
$$\gamma_2\theta_2 = \delta\theta_2\gamma^+\theta_2\gamma\theta_2\delta'\theta_2 = \delta\theta_2\gamma\theta_2\delta'\theta_2 = \gamma_1\theta_2.$$
Let $d$ be the trunk word of $\delta$, $c$ the trunk word of $\gamma$, $d'$ the trunk word of $\delta'$ and $g \in E(X)$ be such that $[g]_C = \gamma^+$. Then by Lemma \ref{lem:thetaprod}
\begin{eqnarray*}
\gamma_2 \theta_1 &=& (\delta\gamma^+\gamma\delta')\theta_1= (\delta\gamma^+\gamma)\theta_1 (dc \cdot \delta'\theta_1)\\
&=& \delta\theta_1 \, (d \cdot (\gamma^+\gamma\theta_1))\,  (dc \cdot (\delta'\theta_1)),\\&=& \delta\theta_1 \, (d \cdot (\gamma^+\theta_1\gamma\theta_1))\,  (dc \cdot (\delta'\theta_1)),
\end{eqnarray*}
and likewise 
\begin{eqnarray*}
\gamma_1 \theta_1 &=&\delta\theta_1 \, (d \cdot (\gamma\theta_1))\,  (dc \cdot (\delta'\theta_1)).
\end{eqnarray*}
Hence, by using a relation from $R_1$  we have
\begin{eqnarray*}
[\gamma_2 \theta_1]_Y  &=& [\delta\theta_1 \, (d \cdot (\gamma^+\theta_1\gamma\theta_1))\,  (dc \cdot (\delta'\theta_1))]_Y\\
&=& [\delta\theta_1 \, (d \cdot (\gamma\theta_1))\,  (dc \cdot (\delta'\theta_1))]_Y\\
&=& [\gamma_1\theta_1]_Y.
\end{eqnarray*}
Similarly, if $\gamma_1 = \delta\gamma\delta'$ and $\gamma_2 = \delta\gamma\gamma^*\delta'$, we again find that $\gamma_2\theta_2 =  \gamma_1\theta_2.$ Let  $d$ be the trunk word of $\delta$, $c$ the trunk word of $\gamma$, $d'$ the trunk word of $\delta'$ and $g \in E(X)$ be such that $[g]_C = \gamma^*$. 
Then using Lemma \ref{lem:thetaprod} again we find
\begin{eqnarray*}
\gamma_1 \theta_1 &=& \delta\theta_1 (d \cdot (\gamma\theta_1)) (dc \cdot (\delta'\theta_1))\\
\gamma_2\theta_1&=& \delta\theta_1 \, (d \cdot (\gamma\theta_1 \,  (c \cdot (\gamma^*\theta_1) ))   \,  (dc \cdot (\delta'\theta_1)).
\end{eqnarray*}
Hence, by using a relation from $R_2$ (noting that $c=\gamma\theta_2$) we have
\begin{eqnarray*}
[\gamma_2 \theta_1]_Y  &=& [\delta\theta_1 \, (d \cdot (\gamma\theta_1 \,  (c \cdot (\gamma^*\theta_1) ))   \,  (dc \cdot (\delta'\theta_1))]_Y\\
&=& [\delta\theta_1 \, (d \cdot (\gamma\theta_1))  \,  (dc \cdot (\delta'\theta_1))]_Y\\
&=& [\gamma_1\theta_1]_Y.
\end{eqnarray*}

This shows that $\hat\theta$ is well-defined. To see that $\hat\theta$ is a monoid morphism, let $\gamma, \delta \in C_X$. Then by Lemma \ref{lem:thetaprod} we have:
\begin{eqnarray*}
[\gamma]_F\hat\theta[\delta]_F\hat\theta &=&([\gamma\theta_1]_Y, \gamma\theta_2)([\delta\theta_1]_Y, \delta\theta_2) \\
&=& ([\gamma\theta_1]_Y \, (\gamma\theta_2 \cdot [\delta\theta_1]_Y), \gamma\theta_2\delta\theta_2)\\
&=& ([\gamma\theta_1 \, (\gamma\theta_2 \cdot (\delta\theta_1))]_Y, (\gamma\delta)\theta_2)\\
&=&([(\gamma\delta)\theta_1]_Y, \gamma\delta\theta_2) \\
&=& [\gamma\delta]_F\hat\theta .
\end{eqnarray*}
It is also clear from definition that the identity element of $C_X /J^\sharp$ is mapped under $\hat\theta$ to $([1]_Y, 1)$.
\end{proof}

It remains to show that  $\hat\theta$ is an embedding. Before proceeding to the proof, we record, for ease of reference, some well-known properties of left Ehresmann monoids (see for example \cite{GG}) that we will need. We omit the left-right dual statements that hold for right Ehresmann monoids.

\begin{lemma}
Let $M$ be a left Ehresmann monoid. Then for all $x,y \in M$ and $e \in E(M)$ we have:
\begin{eqnarray}
\label{eq:(xy)+}(xy)^+ = (xy^+)^+\\
\label{leq:(xy)+}(xy)^+ \leq x^+.
\end{eqnarray}
\end{lemma}

\begin{lemma}\label{lem:simplefold} Let $S$ be a left $E$-Ehresmann monoid. The for any $a,b\in S$ and $e,f\in E$ with  $e\leq f$ we have that $(aeb)^+\leq (afb)^+$. In particular, $(aeb)^+\leq (ab)^+$.\end{lemma}
\begin{proof} We calculate
\[\begin{array}{rcll} (afb)^+(aeb)^+&=&(afb^+)^+(aeb^+)^+&\mbox{by \eqref{eq:(xy)+}}\\
&=&(afb^+)^+(afb^+e)^+& \mbox{as }e\leq f\\
&=&(afb^+e)^+&\mbox{by \eqref{leq:(xy)+}}\\
&=&(aeb)^+&\mbox{by \eqref{eq:(xy)+}}.\end{array}
\]
\end{proof}

\begin{lemma}\label{lem:theta_1} For any $\gamma_1, \gamma_2\in C_X$, if $[\gamma_1\theta_1]_Y=[\gamma_2\theta_1]_Y$,
then $[\gamma_1]_F=[\gamma_2]_F$.
\end{lemma}
\begin{proof}

We first show that if $\gamma\in C_X$ and $\gamma\theta_1$ is related to $\beta\in Z_X$ via a single
application of a relation in $R$, then $\beta=\alpha\theta_1$ for some $\alpha\in C_X$ with 
$[\gamma]_F=[\alpha]_F$. 
The proof then follows from induction on the number of applications of relations from $R$ that we need to obtain
$\gamma_1\theta_2$ from $\gamma_2\theta_2$. 

We denote the trunk of $\gamma$ by $t$, so that by Remark~\ref{rem:theta_1}
we have\[\gamma \theta_1 = \tau_{t,1} \Pi_{p \in P(\gamma)} e_{l(p,\gamma),p}.\]
There are four cases to consider, corresponding to the two types of relation that can be applied and the direction of their application, and recalling that ${\rm Im}\, \theta_1$ is contained in the semilattice $Z_X$.

{\em Case (i)} $\gamma\theta_1=\mu(h\cdot \delta\theta_1)$ and $\mu (h\cdot (\delta^+\theta_1\delta\theta_1)) = \gamma\theta_1 (h\cdot \delta^+\theta_1)=\beta$ for some
$h\in F_X, \delta\in C_X$ and $\mu\in Z_X$.
From Remark~\ref{rem:theta_1} again we have that
\begin{eqnarray*}
\delta \theta_1 &=& \tau_{d,1} \Pi_{p \in P(\delta)} e_{l(p,\delta),p}\end{eqnarray*}
where $d$ is the trunk of $\delta$, 
so that in $Z_X$ we have 
\[\gamma\theta_1 =\tau_{t,1} \Pi_{p \in P(\gamma)} e_{l(p,\gamma),p}=\mu \tau_{d,h}\Pi_{p \in P(\delta)} e_{l(p,\delta),hp} = \mu (h \cdot \delta \theta_1).\]
 Since $\tau_{d,h}$ is a factor of $\tau_{t,1} $, we call upon Lemma~\ref{lem:tauprod} to deduce that
$h\in X^*$ and there exists $u\in X^*$ such that $t=hdu$ and $\tau_{t,1}=\tau_{h,1}\tau_{d,h}\tau_{u,hd}$ where
$\tau_{h,1}$ is interpreted as empty if $h=1$. We may therefore factor $\gamma$ as $\gamma=\eta_1\bar{\delta}\eta_2$, where the
trunks of $\eta_1, \bar{\delta}$ and $ \eta_2\in C_X$ are  $h,d$ and $u$ respectively; if $h=1$ (respectively, $u=1$) we take $\eta_1$
(respectively, $\eta_2$) to be empty and we also insist that $\eta_1$ (if non-empty) finishes with an element of $X^+$ and
$\eta_2$ (if non-empty) starts with an element of $X^+$.

Given the relations in $Q$ that determine $Z_X$, 
we may assume $\mu=\tau_{h,1}\tau_{u, hd}\Pi_{p\in P}e_{f,p}$ where $P\subseteq P(\gamma)$.
Further, for $p\in P(\gamma)$ we have:
\[\begin{array}{lrcll}
\mbox{if  }|p|<|h| \mbox{ or }|p|>|hd|& \mbox{ then }l(p,\gamma)&=&\begin{array}{ll} f & \phantom{mmmm }  \mbox{where } e_{f,p}\mbox{ is a factor of }\mu\end{array}\\
&&\\
\mbox{if }p=hq,|q|\leq |d|& \mbox{then } l(p,\gamma)&=&\left\{ \begin{array}{ll} f&\mbox{where }e_{f,p}\mbox{ is a factor of }\mu\\
&  \mbox{and }q\notin P(\delta)\\
l(q,\delta)& \mbox{where }e_{f,p}\mbox{ not a factor of }\mu\\
&  \mbox{and }q\in P(\delta)\\
fl(q,\delta)& \mbox{where } e_{f,p}\mbox{ is a factor of }\mu\\
&\mbox{and }q\in P(\delta). \end{array}\right.
\end{array}\]

 With the factorisation  $\gamma=\eta_1\bar{\delta}\eta_2$ (and recalling that $\eta_1$ has trunk $h$ and $\bar{\delta}$ has trunk $d$) Lemma~\ref{lem:thetaprod} gives 
\[\begin{array}{rcl}
\mu \tau_{d,h}\Pi_{p \in P(\delta)} e_{l(p,\delta),hp}&=&\gamma\theta_1\\
&=&(\eta_1\theta_1) (h\cdot \bar{\delta}\theta_1)(hd\cdot \eta_2\theta_1)\\
&=&(\eta\theta_1)(\tau_{d,h}\Pi_{p \in P(\bar{\delta})} e_{l(p,\bar{\delta}),hp})(hd\cdot \eta_2\theta_1).\end{array}\]
From the remarks above we see $\delta$ and $\bar{\delta}$ share the same trunk, and if
$\delta$ has an idempotent $l(p,\delta)$ at position $p$, then so does $\bar{\delta}$ 
(and it is the idempotent $l(hp,\gamma)$) and it follows that $l(p,\bar{\delta})\leq l(p,\delta)$.  The fact that 
 $[\bar{\delta}^+]_F\leq [\delta^+]_F$ follows from application of Lemma \ref{lem:simplefold}. Notice also that
\[\eta\theta_1(hd\cdot  \eta_2\theta_1)=\tau_{h,1}\tau_{u, hd}\Pi_{p\in P'(\gamma)} e_{l(p,\gamma),p},\]
where $P'(\gamma)=\{ p\in P(\gamma): |p|<|h| \mbox{ or } |hd|<|p|\}$.

Now let $\alpha=\eta_1\delta^+\bar{\delta}\eta_2\in C_X$. We have that
\[[\alpha]_F=[\eta_1\delta^+\bar{\delta}\eta_2]_F=[\eta_1\delta^+\bar{\delta}^+\bar{\delta}\eta_2]_F=[\eta_1\bar{\delta}^+\bar{\delta}\eta_2]_F=[\eta_1\bar{\delta}\eta_2]_F=[\gamma]_F.\]
Further, making use of Lemma~\ref{lem:thetaprod} and commutativity of $Z_X$ we have
\[\begin{array}{rcl}
\alpha\theta_1&=& (\eta_1\theta_1)(h\cdot (\delta^+\bar{\delta})\theta_1)(hd\cdot \eta_2\theta_1)\\
&=& (\eta_1\theta_1)(h\cdot \delta^+\theta_1)(h\cdot \bar{\delta}\theta_1)(hd\cdot \eta_2\theta_1)\\
&=&(\eta_1\bar{\delta}\eta_2)\theta_1 (h\cdot \delta^+\theta_1)\\
&=&(\gamma\theta_1)(h\cdot \delta^+\theta_1)\\
&=&\beta.
\end{array}\]

{\em Case (ii)} $\gamma\theta_1=\mu(h\cdot \delta\theta_1)$ and $\mu (h\cdot (\delta\theta_1(\delta\theta_2\cdot \delta^*\theta_1)))=\mu (h\cdot \delta\theta_1)(h \, \delta\theta_2\cdot \delta^*\theta_1)))=\beta$ for some
$h\in F_X, \delta\in C_X$ and $\mu\in Z_X$. Arguing exactly as in Case {\em (i)}, we may factorise $\gamma$ as $\gamma = \eta_1\bar{\delta} \eta_2$, where $\delta$ and $\bar{\delta}$ share the same trunk and if $\delta$ has an idempotent $l(p, \delta)$ at position $p$, then $\bar{\delta}$ has an idempotent $l(p, \bar{\delta})$ at position $p$ with $l(p, \bar{\delta}) \leq l(p, \delta)$. By the dual to Lemma \ref{lem:simplefold} we have $[\bar{\delta}^*]_F \leq [\delta^*]_F$. Then setting $\alpha = \eta_1 \bar{\delta}\delta^*\eta_2$, yields
\[[\alpha]_F=[\eta_1\bar{\delta}\delta^*\eta_2]_F=[\eta_1\bar{\delta}\bar{\delta}^*\delta^*\eta_2]_F=[\eta_1\bar{\delta}\bar{\delta}^*\eta_2]_F=[\eta_1\bar{\delta}\eta_2]_F=[\gamma]_F,\]
and
\[\begin{array}{rcl}
\alpha\theta_1&=& (\eta_1\theta_1)(h\cdot (\bar{\delta}\delta^*)\theta_1)(hd\cdot \eta_2\theta_1)\\
&=& (\eta_1\theta_1)(h\cdot (\bar{\delta}\theta_1(\bar{\delta}\theta_2\cdot \delta^*\theta_1)(hd\cdot \eta_2\theta_1)\\
&=& (\eta_1\theta_1)(h\cdot (\bar{\delta}\theta_1)(h\cdot \delta\theta_2\cdot \delta^*\theta_1)(hd\cdot \eta_2\theta_1)\\
&=& (\eta_1\bar{\delta}\eta_1)\theta_1 (h\, \delta\theta_2\cdot \delta^*\theta_1)\\
&=& \gamma\theta_1 (h\,\delta\theta_2\cdot \delta^*\theta_1)\\
&=&\beta.
\end{array}\] 

{\em Case (iii)} $\gamma\theta_1=\mu (h\cdot (\delta^+\theta_1\delta\theta_1))=\mu (h\cdot ((\delta^+\delta)\theta_1))$ and $\mu (h\cdot \delta\theta_1)=\beta$ for some
$h\in F_X, \delta\in C_X$ and $\mu\in Z_X$. The argument here is similar to that in Case $(i)$, but we need slightly more care. As before
\begin{eqnarray*}
\delta \theta_1 &=& \tau_{d,1} \Pi_{p \in P(\delta)} e_{l(p,\delta),p}\end{eqnarray*}
where $d$ is the trunk of $\delta$,
so that in $Z_X$
\[\gamma\theta_1 =  \tau_{t,1} \Pi_{p \in P(\gamma)} e_{l(p,\gamma),p}=\mu e_{\delta^+,h}\tau_{d,h} \Pi_{p \in P(\delta)} e_{l(p,\delta),hp} = \mu (h \cdot (\delta^+ \theta_1 \delta\theta_1)).\]
We again call upon Lemma~\ref{lem:tauprod} to deduce that
$h\in X^*$ and there exists $u\in X^*$ such that $t=hdu$. Again we may  factor $\gamma$ as $\gamma=\eta_1\bar{\delta}\eta_2$, where the
trunks of $\eta_1, \bar{\delta}$ and $ \eta_2\in C_X$ are  $h,d$ and $u$ respectively; if $h=1$ (respectively, $u=1$) we take $\eta_1$
(respectively, $\eta_2$) to be empty and we also insist that $\eta_1$ (if non-empty) finishes with an element of $X^+$ and
$\eta_2$) (if non-empty) starts with an element of $X^+$. In this case note that the first term of $\bar{\delta}$ is in $\mathcal{E}_X$,
that is, $1 \in P(\bar{\delta})$ and $\delta^+l(1, \bar{\delta})= l(1, \bar{\delta})$.

Given the relations in $Q$ that determine $Z_X$, 
we may assume $\mu=\tau_{h,1}\tau_{u,hd}\Pi_{p\in P}e_{f,p}$ where $P\subseteq P(\gamma)$.
Further, for $p\in P(\gamma)$ we have:

\[\begin{array}{lrcll}
\mbox{if }|p|<|h| \mbox{ or }|p|>|hd|& \mbox{ then } l(p,\gamma)&=&\begin{array}{ll} \,\,\, \,\,f & \phantom{mmmm } \mbox{ where } e_{f,p}\mbox{ is a factor of }\mu\end{array}\\
&&\\
&  l(h,\gamma)  = l(1, \bar{\delta})&=&\left\{ \begin{array}{ll} \delta^+&\mbox{ where no }e_{f,h}\mbox{ is a factor of }\mu\\
&  \mbox{ and }1\notin P(\delta)\\
\delta^+f& \mbox{ where }e_{f,h}\mbox{ is a factor of }\mu\\
&  \mbox{ and }1\notin P(\delta)\\
\delta^+l(1,\delta)&  \mbox{ where no }e_{f,h}\mbox{ is a factor of }\mu\\
&  \mbox{ and }
1\in P(\delta)\\
\delta^+fl(1,\delta)& \mbox{ where }e_{f,h}\mbox{ is a factor of }\mu\\
& \mbox{ and } 1\in P(\delta)\end{array}\right.\\
&&\\
\mbox{if }p=hq,\, 1<|q|\leq |d|&\mbox{ then } l(p,\gamma)&=&\left\{ \begin{array}{ll} f&\,\, \,\, \, \mbox{ where }e_{f,p}\mbox{ is a factor of }\mu\\
& \,\, \,\, \,\mbox { and }q\notin P(\delta)\\
l(q,\delta)& \,\, \,\, \,  \mbox{ where no }e_{f,p}\mbox{ is a factor of }\mu\\
& \,\, \,\, \,  \mbox{ and }q\in P(\delta)\\
fl(q,\delta)& \,\, \,\, \, \mbox{ where   }e_{f,p}\mbox{ is a factor of }\mu\\
&\,\, \,\, \, \mbox{ and }q\in P(\delta).\end{array}\right.
\end{array}\]

With the factorisation, $\gamma=\eta_1\bar{\delta}\eta_2$ Lemma~\ref{lem:thetaprod} gives 
\[\begin{array}{rcl}
\mu e_{\delta^+,h}\tau_{d,h} \Pi_{p \in P(\delta)} e_{l(p,\delta),hp}&=&\gamma\theta_1\\
&=&(\eta_1\theta_1) (h\cdot \bar{\delta}\theta_1)(hd\cdot \eta_2\theta_1)\\
&=&(\eta\theta_1)(\tau_{d,h}\Pi_{p \in P(\bar{\delta})} e_{l(p,\bar{\delta}),hp})(hd\cdot \eta_2\theta_1).\end{array}\]
Recalling that $1 \in P(\bar{\delta})$ and $\delta^+l(1, \bar{\delta}) = l(1, \bar{\delta})$ we now let $\bar{\delta}_0\in C_X$ be the same element as  $\bar{\delta}$ 
with the exception that the first idempotent $l(1, \bar{\delta}) = l(h, \gamma)$ is replaced by $l(1, \bar{\delta}_0)$ defined as follows: 
\[\begin{array}{rcl}
l(1,\bar{\delta}_0)&=& \left\{\begin{array}{ll}1&\mbox{ where no }e_{f,h}\mbox{ is  a factor of }\mu\\
& \mbox{ and } 1\notin P(\delta)\\
f& \mbox{ where } e_{f,h}\mbox{ is a factor of }\mu\\
&  \mbox{ and } 1\notin P(\delta)\\
l(1,\delta)&  \mbox{ where no }e_{f,h}\mbox{ is a factor of }\mu\\
&\mbox{ and }1\in P(\delta)\\
fl(1,\delta)& \mbox{ where }e_{f,h}\mbox{ is a factor of }\mu\\
&\mbox{ and }1\in P(\delta). \end{array}\right.\end{array}.\]
 In particular, notice that $\delta^+ \bar{\delta}_0 = \bar{\delta}$.
From the remarks above if $\delta$ has an idempotent $l(p,\delta)$ at position $p$, then so does $\bar{\delta}_0$ and 
$l(p,\bar{\delta}_0)\leq l(p,\delta)$. The fact that 
$[\bar{\delta}_0^+]_F\leq [\delta^+]_F$ follows from application of Lemma \ref{lem:simplefold}. Notice as in Case $(i)$  that
\[\eta\theta_1(hd\cdot  \eta_2\theta_1)=\tau_{h,1}\tau_{u, hd}\Pi_{p\in P'(\gamma)} e_{l(p,\gamma),p},\]
where $P'(\gamma)=\{ p\in P(\gamma): |h|<|p|\mbox{ or } |hd|<|p|\}$.
  Now let $\alpha=\eta_1 \overline{\delta}_0\eta_2\in C_X$. We have that 
\[[\alpha]_F=[\eta_1\bar{\delta}_0\eta_2]_F=[\eta_1\bar{\delta}_0^+\bar{\delta}_0\eta_2]_F= [\eta_1\delta^+\bar{\delta}^+_0\bar{\delta}_0\eta_2]_F=[\eta_1{\delta}^+\bar{\delta}_0\eta_2]_F= [\eta_1\bar{\delta}\eta_2]_F=[\gamma]_F.\]
Further, making use of Lemma~\ref{lem:thetaprod} we have 
\[
\alpha\theta_1=(\eta_1\theta_1)(h\cdot \bar{\delta}_0\theta_1)(hd\cdot \eta_2\theta_1).\]
We claim that
\[(\eta_1\theta_1)(h\cdot \bar{\delta}_0\theta_1)(hd\cdot \eta_2\theta_1)=\mu (h\cdot \delta\theta_1)\]
so that $\alpha\theta_1=\beta$. 
First, note that since $\alpha = \eta_1 \bar{\delta}_0 \eta_2$ and $\gamma = \eta_1 \bar{\delta}\eta_2$, where $\eta_1$ has trunk word $h$ and $l(\bar{\delta}_0, p) = l(\bar{\delta}, p)$ for all $p \neq 1$, we have that $\alpha\theta_1$ and $\gamma\theta_1$ have the same factors of the
form $y_{x,k}$ and, with the possible exception of $k=h$, of the form $e_{g,k}$. Moreover, since $\gamma\theta_1 = \mu(h \cdot (\delta^+\theta_1 \delta\theta_1))$ these factors are the same as those of  $\mu (h\cdot \delta\theta_1)$.
It remains only to examine any factors of $\alpha\theta_1$ and of $\mu (h\cdot \delta\theta_1)$ of the form $e_{g,h}$. From the definition of $\alpha$ we see that if $\alpha\theta_1$
has a factor $e_{g,h}$, then this is unique, and it exists if and only if the element
$\overline{\delta}_0$ of $C_X$ begins with a non-identity idempotent, that is, it begins with an element of $\mathcal{E}_X$. This is the case if and only if $\mu$ has a factor of the form $e_{f,h}$ or $1\in P(\delta)$, and these are precisely the conditions which give that
$\mu (h\cdot \delta\theta_1)$ has a factor of the form $e_{g,h}$. Under this assumption, the factor  $e_{g,h}$ of $\alpha\theta_1$ is such that $g=f$ ($g=l(1,\delta)$, $g=fl(1,\delta)$, respectively) where $e_{f,h}$ is a factor of $\mu$ and $1\notin P(\delta)$ (no $e_{f,h}$ is  a factor of $\mu$ and $1\in P(\delta)$,  $e_{f,h}$ is a factor of $\mu$ and $1\in P(\delta)$, respectively). But then, after potentially applying relations from $Q$, we have that $g$ is precisely the idempotent such that $e_{g,h}$ is the unique factor of
$\mu (h\cdot \delta\theta_1)$ of the form $e_{g',h}$. 

{\em Case (iv)}  $\gamma\theta_1=\mu (h\cdot (\delta\theta_1(\delta\theta_2\cdot \delta^*\theta_1)))$
and $\beta=\mu(h\cdot \delta\theta_1)$ for some
$h\in F_X, \delta\in C_X$ and $\mu\in Z_X$. 
This is dual to Case {\em (iii)}. (Briefly, letting $d$ be the trunk of $\delta$, we find that $\gamma$ may be factored as $\gamma = \eta_1\bar{\delta} \eta_2$ for some $\bar{\delta}$ with trunk $d$ and such that $d \in P(\bar{\delta})$ and $l(d, \bar{\delta})\delta^* = l(d, \bar{\delta})$. Then taking $\bar{\delta}_0 \in C_X$ to be the element obtained from $\bar{\delta}$ by `removing' the factor $\delta^*$ we have $\bar{\delta}_0\delta^* = \bar{\delta}$. Moreover, by construction we find that if $\delta$ has an idempotent $l(p, \delta)$ at position $p$, then $\bar{\delta}_0$ has an idempotent $l(p, \bar{\delta}_0)$ at this position where $l(p, \bar{\delta}_0) \leq l(p, \delta)$, from which we deduce (using the dual to Lemma \ref{lem:simplefold}) that $[\delta_0^*]_F \leq [\delta^*]F$. Then taking $\alpha = \eta_1 \bar{\delta}_0\eta_2 \in C_X$, one finds that $[\alpha]_F = [\gamma]_F$ and $\alpha \theta_1 = \beta$.)
\end{proof}

\begin{thm}\label{thm:two-sided}
The free Ehresmann monoid embeds (as a monoid) into the semidirect product $Y_X \rtimes F_X$.
\end{thm}

\begin{proof}
It suffices to show that $\hat\theta:C_X /J^\sharp\rightarrow Y_X \rtimes F_X$ is injective. To this end, suppose that $\gamma_1, \gamma_2\in C_X$ and
$[\gamma_1]_F\hat\theta=[\gamma_2]_F\hat\theta$. By definition of $\hat\theta$ we have
\[([\gamma_1\theta_1]_Y, \gamma_1\theta_2) = ([\gamma_2\theta_1]_Y, \gamma_2\theta_2),\]
so that immediately from Lemma~\ref{lem:theta_1} we obtain $[\gamma_1]_F=[\gamma_2]_F$.
\end{proof}

\begin{remark}
Notice that with the above embedding, for $x\in X$ we have that
\[([x]_F\hat{\theta})^+=([y_{x,1}]_Y,x)^+=([y_{x,1}]_Y,1)\]
whereas
\[([x]_F^+)\hat{\theta}=([x^+]_F)\hat{\theta}=([e_{x^+,1}]_Y,1)\]
so that we see explicitly that 
\[([x]_F \hat{\theta})^+\neq ([x]_F^+)\hat{\theta},\]
and hence $\hat\theta$ is not an embedding of bi-unary monoids.
\end{remark}

\begin{remark}
\label{rem:alt}
Note that the embedding given by Theorem \ref{thm:two-sided} allows us to give an alternative proof that $\FAd(X)$ (or indeed, $\FLAd(X)$) is not left coherent for $X = \{x_1, x_2, \ldots \}$, by appealing to Corollary \ref{cor:pi}, taking $Y = Y_X$, $G = F_X$, $M = \FAd(X)$ viewed now as a submonoid of $Y_X \rtimes F_X$, with $a=(y_{x_1, 1},x_1), b=(y_{x_2,1},x_2), e_i = (e_{f_i, 1}, 1) \in M$ for $f_i:=(x_2 x_1^i)^+ \in \mathcal{E}_X$. 
\end{remark}

\section{Questions}
\label{sec:questions}
We have exhibited several examples of $E$-unitary inverse monoids constructed from groups where our technique may be applied. Notably we have seen that the Szendrei and Margolis-Meakin expansions of {\em certain groups} are not (left or right) coherent. This leaves open the following question:

\begin{question}
Which Szendrei expansions ${\rm Sz}(G)$ {\em are} coherent? Which Margolis-Meakin expansions $M(G, X)$ {\em are} coherent?
\end{question}

Our result also applies to show that the free objects in certain (quasi-)varieties are not coherent as monoids. We summarise the situation for left and right coherency of free objects in the (quasi-) varieties of monoids considered here and elsewhere in the literature.

\medskip
\noindent
\begin{tabular}{|r| l| l |}
\hline
&Left coherent & Right coherent \\
\hline
free group & yes \cite{G}& yes\cite{G}\\
free monoid & yes \cite{GHR}& yes \cite{GHR}\\
free inverse& for rank $>1$, no \cite[Theorem 7.5]{GH}& for rank $>1$, no \cite[Theorem 7.5]{GH}\\
free ample& for rank $>1$, no \cite[Theorem 7.5]{GH} & for rank $>1$, no \cite[Theorem 7.5]{GH}\\
free left ample&  for rank $>1$, no \cite[Theorem 7.5]{GH} & for all ranks, yes \cite[Theorem 5.7]{GH}\\
free Ehresmann&  for rank $>1$, no Theorem \ref{cor:adequate} & for rank $>1$, no Theorem \ref{cor:adequate}\\
free left Ehresmann&  for rank $>1$, no Theorem \ref{cor:adequate} & ?\\
&(see also Remark \ref{rem:alt})& \\
\hline
\end{tabular}

\medskip
Given the positive result for the free left ample monoid, one is immediately drawn to ask whether a similar result may hold for the free left Ehresmann monoid.
\begin{question}\label{qn:hard}
Is the free left Ehresmann monoid right coherent in all ranks?
\end{question}
We remark that a positive answer to Question~\ref{qn:hard}
is likely to be technically challenging, given the nature of the existing proof the fact that free left ample monoids are right coherent.

Coherency of the {\em monogenic} free objects in the table above has not been explicitly considered in the literature. However, by appealing to several known results, it is straightforward to show that the free monogenic inverse monoid is coherent:
\begin{example}
We show that the free monogenic inverse monoid, $F:=\FI(x)$, is right (and hence also left) coherent. By an abuse of notation, for integers $l,r$ with $l \leq r$ let us write $[x^l,x^r]$ to denote the set of elements $\{x^i: l \leq i \leq r\}$. The elements of $F$ are then the pairs of the form $([x^l,x^r],x^m)$, where $l \leq m \leq r$ and $-l,r \geq 0$, and the idempotents are the elements with $m=0$. Recall that the natural partial order on idempotents is determined by $e \leq f$ if and only if $ef=e$. Thus for $e=([x^l, x^r], 1)$ and $f=([x^p,x^q], 1)$ we see that $e \leq f$ if and only if $[p,q] \subseteq [l,r]$. Clearly the poset $(E(F), \leq)$ satisfies the ascending chain condition. Notice also that there are no infinite antichains of idempotents in this partial order (for an idempotent $e = ([x^l,x^r],1)$ any element incomparable to $e$ must lie in exactly one of the following (finitely many) sets: $L_i = \{f: f = ([x^i,x^q],1) \mbox{ for some } q>r\}$ for $l<i \leq 0$ and $R_j = \{f: f = [x^p,x^{j}] \mbox{ for some } p<l\}$ for $0 \leq j < r$, but notice that any two elements of $L_i$ are comparable, and any two elements of $R_i$ are comparable, so there is no infinite antichain containing $e$). Since $F$ is inverse we see that the principal right ideals of $F$ are precisely the sets $eF$ where $e\in E(F)$, and it is easy to see that if $e,f \in E(F)$, we have $eF \subseteq fF$ if and only if $e \leq f$, from which it is then clear that the set of principal right ideals satisfies the ascending chain condition. It now follows from \cite[Theorem 3.2]{M21} that $F$ is weakly right Noetherian, and since by \cite[Corollary 3.3]{DGHRZ20} any weakly right noetherian regular monoid is right coherent, we obtain that $F$ is right coherent. 
\end{example}
This leaves the remaining questions:

\begin{question}
Is the free ample/left ample/Ehresmann/left Ehresmann monoid of rank $1$ left or right coherent?
\end{question}

Where a `no' answer is obtained in the previous table above, a natural question is whether the corresponding monoid satisfies a weaker finiteness condition, such as  weakly left or right coherent. We summarise what is known in this regard in the table below:

\medskip
\noindent\begin{tabular}{|r| l| l |}
\hline
&Weakly left coherent & Weakly right coherent\\
\hline
free inverse& yes  by \cite[Corollary 3.3]{BGR23} & yes by \cite[Corollary 3.3]{BGR23}\\
free ample& yes by \cite[Corollary 3.12]{CG21} and \ref{prop:rabundweak} &  yes by \cite[Corollary 3.12]{CG21} and \ref{prop:rabundweak}\\
free left ample& yes by \cite[Corollary 3.12]{CG21} and \ref{prop:rabundweak} & yes [it is right coherent]\\
free left Ehresmann& yes by Theorem \ref{thm:FLEisweaklycoherent} & yes by Theorem \ref{thm:FLEisweaklycoherent}\\
free Ehresmann& (finitely left equated by \ref{prop:rabundweak}) & (finitely right equated by \ref{prop:rabundweak})\\
\hline
\end{tabular}

\medskip
This leaves open the following question:

\begin{question}
Is the free Ehresmann monoid  weakly left or weakly right coherent?
\end{question}

Finally, in a different direction, given the role of special semidirect products of the form $\mathcal{S}(G)$ where $G$ is a group in many of the aforementioned results, we propose a general study of coherence properties of monoids of the form $\mathcal{S}(M)$ where $M$ is a monoid.

\begin{question}
For which monoids $M$ is $\mathcal{S}(M)$ (weakly) left/right coherent?
\end{question}

\begin{question}
Can a monoid containing a non-coherent monoid $\mathcal{S}(M)$ as a subsemigroup be left or right coherent?
\end{question}

\end{document}